\documentclass[11pt]{amsart}
\usepackage{pdfsync}
\usepackage{epsfig}
\usepackage{amsmath, amssymb, euscript,  enumerate}
\usepackage{pxfonts}
\usepackage{color}
\usepackage{verbatim}
\newcommand{\R}{{\mathbb R}}

\newcommand{\supp}{\text{\rm supp}}

\newcommand{\ap}{\alpha}             
\newcommand{\bt}{\beta}
\newcommand{\gm}{\gamma}             \newcommand{\Gm}{\Gamma}
\newcommand{\dt}{\delta}             
\newcommand{\vep}{\varepsilon}

\newcommand{\kp}{\kappa}

\newcommand{\ld}{\lambda}            \newcommand{\Ld}{\Lambda}
\newcommand{\sm}{\sigma}             
\newcommand{\vp}{\varphi}
\newcommand{\om}{\omega}             \newcommand{\Om}{\Omega}
\newcommand{\vr}{\varrho}            \newcommand{\iy}{\infty}
            
\newcommand{\f}{\frac}

\newcommand{\fB}{{\mathfrak B}}

\newcommand{\fL}{{\mathfrak L}}

\newcommand{\fN}{{\mathfrak N}}

\newcommand{\BN}{{\mathbb N}}

\newcommand{\BR}{{\mathbb R}}

\newcommand{\cA}{{\mathcal A}}
\newcommand{\cB}{{\mathcal B}}
\newcommand{\cC}{{\mathcal C}}

\newcommand{\cG}{{\mathcal G}}

\newcommand{\cI}{{\mathcal I}}
\newcommand{\cJ}{{\mathcal J}}
\newcommand{\cK}{{\mathcal K}}
\newcommand{\cL}{{\mathcal L}}
\newcommand{\cM}{{\mathcal M}}

\newcommand{\s}{\setminus}         \newcommand{\ep}{\epsilon}
\newcommand{\n}{\nabla}            \newcommand{\e}{\eta}
\newcommand{\pa}{\partial}        \newcommand{\fd}{\fallingdotseq}
       
    \newcommand{\ds}{\displaystyle}

 \newcommand{\pf }{\noindent{\it Proof. }}
\newcommand{\rk }{\noindent{\it Remark. }}

\newcommand{\aee }{\text{\rm a.e.}} \newcommand{\diam }{\text{\rm diam}}
  \newcommand{\pv }{\text{\rm p.v.}}
\newcommand{\osc }{\text{\rm osc}}  \newcommand{\dd }{\text{\rm d}}

\newcommand{\rB }{{\text{\rm B}}}   \newcommand{\rC }{{\text{\rm C}}}

\newcommand{\rQ}{{\text{\rm Q}}}
\newcommand{\rK }{{\text{\rm K}}}
\newcommand{\bQ}{\bar{\text{\rm Q}}}

\newtheorem{thm}[subsubsection]{Theorem}
\newtheorem{lemma}[subsubsection]{Lemma}
\newtheorem{cor}[subsubsection]{Corollary}

\newtheorem{definition}[subsubsection]{Definition}

\numberwithin{equation}{subsection}

\title[fully nonlinear integro-differential operators]{ Regularity results for fully nonlinear parabolic integro-differential operators
 }
\author{ Yong-Cheol Kim and Ki-Ahm Lee }
\begin{document}
\begin{abstract}
 In this paper, we consider the regularity theory for fully nonlinear parabolic integro-differential
 equations with symmetric kernels. We are able to find
 parabolic  versions of Alexandro詮�-Backelman-Pucci estimate
 with $0<\sigma<2$.  And we show a Harnack inequality, H\"older regularity, and
 $C^{1,\alpha}$-regularity of the solutions by obtaining decay
 estimates of their level sets.
\end{abstract}

\address{$\bullet$ Yong-Cheol Kim : Department of Mathematics Education, Korea University, Seoul 136-701,
Korea }

\email{ychkim@korea.ac.kr}

\address{$\bullet$ Ki-Ahm Lee : Department of Mathematics, Seoul National University, Seoul 151-747,
Korea} \email{kiahm@math.snu.ac.kr}

\maketitle

\tableofcontents

\section{Introduction}\label{sec-intro}

In this paper, we are going to find a parabolic version of Alexandro詮�-Backelman-Pucci estimate, Harnack inequality, H\"older regularity, and $C^{1,\alpha}$-regularity whose elliptic versions have been considered
at \cite{CS} with symmetric kernels and at \cite{KL1,KL2} with nonsymmetric kernels.
The concept of viscosity solutions and notations are parallel with those at \cite{CS,KL1,KL2} with minor changes.

The linear {\it parabolic integro-differential operators}  are given as
\begin{equation}\label{eq-nonsym-l}
\cL u(x,t)-\partial_t u(x,t)=\pv\int_{\BR^n}\mu(u,x,y)K(y,t)\,dy-\partial_t u(x,t)
\end{equation}
for  $\mu(u,x,y,t)=u(x+y,t)-u(x,t)-(\n u(x,t)\cdot y)\chi_{B_1}(y)$, which
describes  the infinitesimal generator of given purely jump
processes, i.e. processes without diffusion or drift part \cite{CS}.
We refer the detailed definitions of notations to \cite{CS, KL1, KL2}. Then we
see that $\cL u(x,t)$ is well-defined provided that
$u\in\rC_x^{1,1}(x,t)\cap\rB(\BR^n\times [0,T])$ where $\rB(\BR^n\times [0,T])$ denotes {\it the
family of all real-valued bounded functions defined on $\BR^n\times [0,T]$} and $\rC_x^{1,1}(x,t)$ means $C^{1,1}$-function at $x$ for a given $t$.  If
$K$ is symmetric (i.e. $K(-y,t)=K(y,t)$), then an odd
   function $\bigl[(\n u(x,t)\cdot
y)\chi_{B_1}(y)\bigr]K(y,t)$ will be canceled in the integral, and so
we have that
$$\cL  u(x,t)=\pv\int_{\BR^n}\bigl[u(x+y,t)+u(x-y,t)-2u(x,t)\bigr]K(y,t)\,dy .$$

Nonlinear integro-differential operators come from the  stochastic
control theory related with
\begin{equation*}
\cI u(x,t)=\sup_{\ap}\cL_{\ap}u(x,t),
\end{equation*}
or game theory associated with
\begin{equation}\label{eq-game}
\cI u(x,t)=\inf_{\bt}\sup_{\ap}\cL_{\ap\bt}u(x,t),
\end{equation}
when the stochastic process is of L\`evy type allowing jumps; see
\cite{S, CS, KL1}. Also an operator like $\cI
u(x,t)=\sup_{\ap}\inf_{\bt}\cL_{\ap\bt}u(x,t)$ can be considered.
Characteristic properties of these operators can easily be derived
as follows;
\begin{equation}\label{eq-class-1}
\begin{split}\inf_{\ap\bt}\cL_{\ap\bt}
v(x,t)&\le\cI [u+v](x,t)-\cI u(x,t)\le\sup_{\ap\bt}\cL_{\ap\bt} v(x,t).
\end{split}
\end{equation}

\subsection{Operators}
In this section, we introduce a class of operators. All notations
and the concepts of viscosity solution follow \cite{CS} with minor changes.

For parabolic setting and our purpose, we shall consider functions
$u(x,t)$ defined on $\BR^{n}\times [0,T]$ and restrict our attention to the
operators $\cL$ where the measure  is given by a positive
kernel $K$ which is symmetric. That is to say, the
operators $\cL$ are given by
\begin{equation}\label{eq-sym-l}
\cL u(x,t)=\pv \int_{\BR^n}\mu(u,x,y,t)K(y,t)\,dy
\end{equation}
where $\mu(u,x,y,t)=u(x+y,t)+u(x-y,t)-2 u(x,t)$. And we consider the
class $\fL$ of the operators $\cL$ associated with  positive kernels $K\in\cK_0$ satisfying that
\begin{equation}\label{eq-kernel-bound}(2-\sm)\f{\ld}{|y|^{n+\sm}}\le
K(y,t)\le(2-\sm)\f{\Ld}{|y|^{n+\sm}},\,\,0<\sm<2.
\end{equation}
The maximal operator and the minimal operator with respect to $\fL$
are defined by
\begin{equation}\label{eq-pucci}\cM^+_{\fL}u(x,t)=\sup_{\cL\in\fL}\cL u(x,t)\,\,\text{ and }\,\,\cM^-_{\fL}u(x,t)=\inf_{\cL\in\fL}\cL u(x,t).
\end{equation}

In what follows, we let $\Om\subset\BR^n$ be a bounded open domain,
$I=(\tau_1,\tau_2]$ be a bounded half-open interval where
$\tau_1<-100$ and $\tau_2>100$, and $J=(a,b]\subseteq I$. For
$(x,t)\in\Om_J\fd\Om\times J$ and a function $u:\BR^n\times I\to\BR$
which is semicontinuous on $\overline\Om_J$, we say that $\vp$
belongs to the function class $\rC^2_{\Om_J}(u;x,t)^+$ (resp.
$\rC^2_{\Om_J}(u;x,t)^-$) and we write
$\vp\in\rC^2_{\Om_J}(u;x,t)^+$ (resp.
$\vp\in\rC^2_{\Om_J}(u;x,t)^-$) if there exists a
$U_{t,\dt}$ such that $\vp(x,t)=u(x,t)$ and $\vp>u$
(resp. $\vp<u$) on $U_{t,\dt}\s\{(x,t)\}$ for some open neighborhood
$U\subset\Om$ of $x$ and some $(t-\dt,t]\subset J$, where
$U_{t,\dt}=U\times(t-\dt,t]$. We note that geometrically $u-\vp$
having a local maximum at $(x,t)$ in $\Om_J$ is equivalent to
$\vp\in\rC^2_{\Om_J}(u;x,t)^+$ and $u-\vp$ having a local minimum at
$(x,t)$ in $\Om_J$ is equivalent to $\vp\in\rC^2_{\Om_J}(u;x,t)^-$.
And  the expression for
$\cL_{\ap\bt}\,u(x,t)$ and $\cI u(x,t)$ may be written
as
\begin{equation*}\begin{split}\cL_{\ap\bt}\,u(x,t)&=\int_{\BR^n}\mu(u,x,y,t)K_{\ap\bt}(y,t)\,dy,\\
\cI
u(x,t)&=\inf_{\bt}\sup_{\ap}\cL_{\ap\bt}\,u(x,t),\end{split}\end{equation*}
where $K_{\ap\bt}\in\cK_0$. Then we see
$\cM^-_{\fL}u(x,t)\le\cI
u(x,t)\le\cM^+_{\fL}u(x,t)$, and
$\cM^+_{\fL}u(x,t)$ and $\cM^-_{\fL}u(x,t)$ have the
following simple forms;
\begin{equation}\label{eq-pucci-integral}
\begin{split}
&\cM^+_{\fL}u(x,t)=(2-\sm)\int_{\BR^n}\f{\Ld\mu^+(u,x,y,t)
-\ld\mu^-(u,x,y,t)}{|y|^{n+\sm}}\,dy,\\
&\cM^-_{\fL}u(x,t)=(2-\sm)\int_{\BR^n}\f{\ld\mu^+(u,x,y,t)
-\Ld\mu^-(u,x,y,t)}{|y|^{n+\sm}}\,dy,
\end{split}
\end{equation}
where $\mu^+$ and $\mu^-$ are given by
\begin{equation*}\begin{split}\mu^{\pm}(u,x,y,t)&=\max\{\pm
    \mu(u,x,y,t),0\}.
\end{split}\end{equation*}

A function $u:\BR^n\times I\to\BR$ is said to be $\text{\rm
$\rC^{1,1}_{x,\pm}$ at $(x,t)\in\BR^n\times I$}$ $($we write
$u\in\rC^{1,1}_{x,\pm}(x,t)$$)$, if there are $r_0>0$ and $M>0$ (independent of $s$)
such that
\begin{equation}\label{eq-u-quadradic}
\pm\left(u(x+y,t)+u(x-y,t)-2u(x,t)\right)\le M\,|y|^2
\end{equation}
for any
$(y,s)\in B_{r_0}(0)\times(-r_0,0].$

 We write $u\in\rC_{x,\pm}^{1,1}(U_t)$
if $u\in\rC_{x,\pm}^{1,1}(x,t)$ for any $(x,t)\in U_t$ and the constant
$M$ in \eqref{eq-u-quadradic} is independent of $(x,t)$, where
$U_t=U\times(t-\dt,t]\subset\BR^n\times I$ for some $\dt>0$ for an
open subset $U$ of $\BR^n$. And we denote
$\rC_x^{1,1}(x,t)=\rC^{1,1}_{x,+}(x,t)\cap\rC^{1,1}_{x,-}(x,t)$,
$\rC_x^{1,1}(U_t)=\rC^{1,1}_{x,+}(U_t)\cap\rC_{x,-}^{1,1}(U_t)$, and $\rC^{1,1}(U_t)=\rC^{1,1}_{x,}(U_t)\cup \rC^{0,1}_{t}(U_t)$.

We note that if $u\in\rC_x^{1,1}(x,t)$, then $\cI u(x,t)$ and $\cM^{\pm}_{\fL}u(x,t)$ will be well-defined. We
shall use these maximal and minimal operators to obtain regularity
estimates.

Let $K(x,t)=\sup_{\ap}K_{\ap}(x,t)$ where $K_{\ap}$'s are all the
kernels of all operators in a class $\fL$. For any class $\fL$, we
shall assume that
\begin{equation}\label{eq-kernel-cond}
\int_{\BR^n}(|y|^2\wedge 1)\,K(y,t)\,dy<\iy.
\end{equation}

The following is a kind of
operators of which the regularity result shall be obtained in this
paper.

\begin{definition}\label{def-class} Let $\fL$ be a class of linear
integro-differential operators. Assume that \eqref{eq-kernel-cond}
holds for $\fL$. Then we say that an operator $\cJ$ is
elliptic with respect to $\fL$, if it satisfies the following
properties:

$(a)$ $\cJ u(x,t)$ is well-defined for any $u\in\rC_x^{1,1}(x,t)\cap
\rB(\BR^n\times {t})$.

$(b)$ $\cJ u$ is continuous on an open $\Om_J\subset\BR^n\times I$,
whenever $u\in\rC_x^{1,1}(\Om_J)\cap\rB(\BR^n\times I)$.

$(c)$ If $u,v\in\rC_x^{1,1}(x,t)\cap\rB(\BR^n\times {t})$, then we
have that
\begin{equation}\label{eq-class}
\cM^-_{\fL}[u-v](x,t)\le\cJ u(x,t)-\cJ
v(x,t)\le\cM^+_{\fL}[u-v](x,t).
\end{equation}
\end{definition}

The concept of viscosity solutions, its comparison principle and
stability properties can be obtained with small modifications from \cite{CS}  as
\cite{W1}.
We summarized them at section \ref{sec-pre}.

\subsection{Main equation}
The natural Dirichlet problem for such parabolic  nonlocal operator $\cI$ in $\BR^n\times (0,T]$ is given as the following.
Given  functions $\vp$ and $g$ defined on $(\BR^n\times (0,T])\s (\Omega\times (0,T])$ and $\Omega$ respectively, we want to
find a function $u$ such that
\begin{equation}
\begin{cases}
u_t(x,t)-\cI u(x,t)=0 &\text{ for any $(x,t)\in\Om\times (0,T] $,}\\
u(x,t)=\vp(x,t) &\text{for $(x,t)\in(\BR^n\times [0,T])\s (\Omega\times [0,T])$.}\\
u(x,0)=g(x)&\text{for $x\in\Omega$.}
\end{cases}
\end{equation}
 Note that the
boundary condition is given not only on $\pa_p(\Omega\times (0,T])$ but also on the
whole complement of $(\Omega\times (0,T])$. This is because of the nonlocal
character of the operator $\cI$. From the stochastic point of view,
it corresponds to the fact that a discontinuous L\`evy process can
exit the domain $(\Omega\times (0,T])$ for the first time jumping to any point in
$(\BR^n\times (0,T])\s (\Omega\times (0,T])$.

In this paper, we shall concentrate mainly upon the regularity
properties of viscosity solutions to an equation $u_t(x,t)-\cI
u(x,t)=0$.

\subsection{Known results and Key Observations}

There are some known results about Harnack inequalities and H\"older
estimates for  integro-differential operators with positive symmetric
kernels (see \cite{J} for analytical proofs and \cite{BBC},
\cite{BK1}, \cite{BK2},\cite{BL}, \cite{KS}, \cite{SV} for
probabilistic proofs).
 More general results for the elliptic cases have been shown\cite{CS}) for symmetric kernels and \cite{KL1,KL2} for nonsymmetric kernels.
The analytic approach for the linear parabolic equations can be found at \cite{CV}.

There are some serious difficulties arises when we try to extend the results in elliptic case to the parabolic
equations.
Key observations are the following:

$\bullet$  The equation is local in time while it is nonlocal in the space variable.
                 Caffarelli and Silvestre considered a sequence of  dyadic rings in space at A-B-P estimate
                 to find the balance of quantities in the integral. But a simple generalization of the ring in space
to one in space-time fails since the equation is local in the time variable. Such unbalance between local and nonlocal terms in the equation requires more fine analysis to find a parabolic version of A-B-P estimate at section \ref{sec-abp}.

$\bullet$ There is a time delay to control the lower bound in a small neighborhood of a point by the current value at the point, which is a main difference between elliptic and parabolic equations.
Such  time-delay effect has been shown at Lemma \ref{lem-decay-key} with a parabolic A-B-P estimate and a barrier, Lemma \ref{lem-barrier-3}. And we also need a parabolic version of Calderon-Zygmund decomposition which has been considered at \cite{W1}.

\subsection{Outline of Paper}
In Section \ref{sec-abp},  we show   nonlocal versions of the parabolic
nonlocal Alexandroff-Backelman-Pucci estimate  to handle the difficulties
caused by the locality in the time varibale.  In Section
\ref{sec-decay}, we construct a special function and apply a parabolic version of A-B-P
estimates  to obtain the decay estimates of upper level sets which is
essential in proving H\"older estimates in Section \ref{sec-regularity}.

In Section \ref{sec-regularity}, we prove the H\"older estimates and
an interior $\rC^{1,\ap}$-estimates come from the arguments at
\cite{CS,KL1,KL2}. We also show a Harnack inequality.

\subsection{Notations}
We summarize the notations of domains  briefly for the reader's convenience.
\begin{enumerate}
\item  $\pa_p(\Om\times J)=\pa\Om\times J\bigcup\overline\Om\times\{a\}$. $\pa^*_p(\Om\times J)=(\R^n\s\Om)\times J\bigcup\R^n\times\{a\}$.
\item  $Q_r=B_r(0)\times (-r^{\sigma},0]$, $Q_r(x,t)=Q_r+(x,t)$, and $Q=Q_1(0,1)$, $Q^-_r=Q_r(0,r^{\sigma})$, $Q^{+}_r=Q_r(x,t)$.
\item $\rQ_1\fd
Q=B_1\times(0,1]$, $\rQ_r\fd Q_r(0,r^{\sm})=B_r\times(0,r^{\sm}]$,
$\rQ^-_r\fd\rQ_r+(0,r^{\sm})=B_r\times(r^{\sm},2 r^{\sm})$,
$\rQ^+_r\fd\rQ^-_r+(0,2 r^{\sm})=B_r\times(3 r^{\sm},4 r^{\sm})$,
$\rQ_r(x_0,t_0)=\rQ_r+(x_0,t_0)$ and
$\rB^{\dd}_r(x_0,t_0)=\{(x,t)\in\BR^n\times
I:\dd\bigl((x,t),(x_0,t_0)\bigr)<r\}$ are defined at Section \ref{sec-decay-it}
\item $\bQ_r\fd\rQ_r\cup Q_r=B_r\times(-r^{\sm},r^{\sm}]$ and
$\bQ_r(x_0,t_0)=\bQ_r+(x_0,t_0)$ are defined at Section
\ref{sec-regularity-c1a}.
\item   $K_r=(-r,r)^n\times (-r^2,0]$ and
$K_r(x,t)=K_r+(x,t)$.
\item  $\rK_r=(-r,r)^n\times (-r^{\sigma},0]$,
 $\rK_r(x,t)=\rK_r+(x,t)$,
   $\rK^-_r=\rK_r(0,r^{\sigma})$, $\rK^+_{3r}=(-3 r,3 r)^n\times(r^{\sigma},(3^{\sm}+2)
r^{\sigma}]$,  $\rK_r^-(x,t)=\rK_r^-+(x,t)$,  and $\rK_{3r}^+(x,t)=\rK_{3r}^+ +(x,t)$.
\item $\cG^u_h$ and $\cB^u_s$ are defined at Section \ref{sec-decay-key}.
\item $ \overline{\rK}^m$,   $\tilde{\rK}^m$,    $\cA^m_{\delta}$,  and $\fB^m_{\delta}$ are defined at Section \ref{sec-decay-cz}.

\end{enumerate}


\section{ Preliminaries }\label{sec-pre}

The parabolic distance for $P_1=(x,t)$ and $P_2=(y,s)$ is defined to
be
\begin{equation}
d(P_1,P_2)=
\begin{cases}
(|x-y|^{\sigma}+|t-s|)^{1/\sigma}, &t\le s,\\
\iy, &t>s.
\end{cases}
\end{equation}
We define the parabolic boundary of $\Omega\times J$ by
$\pa_p\Omega\times J=\pa\Om\times J\bigcup\overline\Om\times\{a\}$. For
$r>0$, we set $Q_r(x,t)=B_r(x)\times (t-r^\sigma,t]$ and
$Q^c_r(x,t)=(\BR^n\s B_r(x))\times (t-r^\sigma,t]$. We also define
the diameter $d$ of $Q_r(x,t)$ by $d=\sqrt 5\,r$.

\begin{definition} \label{def-viscosity}Let $f:\BR^n\times I\to\BR$ be a continuous
function. Then a function $u:\BR^n\times I\to\BR$ which is upper
$($lower$)$ semicontinuous on $\overline\Omega\times J$ is said to be a
viscosity subsolution (viscosity supersolution) of an equation
$u_t-\cJ u=f$ on $\Omega\times J$ and we write $u_t-\cJ u\le f$ $($ $u_t-\cJ
u\ge f$ $)$ on $\Omega\times J$ in the viscosity sense, if for any
$(x,t)\in\Omega\times J$ there is some neighborhood $Q_r(x,t)\subset\Omega\times J$ of
$(x,t)$ such that $u_t(x,t)-\cJ u(x,t)$ is well-defined and
$\vp_t(x,t)-\cJ v(x,t)\le f(x,t)$ $($ $\vp_t(x,t)-\cJ v(x,t)\ge
f(x,t)$ $)$ for $v=\vp\chi_{Q_r(x,t)}+u\chi_{Q_r^c(x,t)}$ whenever
$\vp\in\rC^2(Q_r(x,t))$ with $\vp(x,t)=u(x,t)$ and $\vp>u$ $($
$\vp<u$ $)$ on $Q_r(x,t)\s\{(x,t)\}$ exists. Also a function $u$is called as a viscosity solution if 
it is both a viscosity subsolution and a viscosity supersolution
to $u_t-\cJ u=f$ on $\Omega\times J$.
\end{definition}

\begin{thm}\label{thm-viscosity}
Let $f:\BR^n\times I\to\BR$ be a function. Then we have the
followings:

$(a)$ If $u:\BR^n\times I\to\BR$ is a function which is upper
semicontinuous on $\overline\Omega\times J$, then $u_t-\cI u\le f$ on $\Omega\times J$
in the viscosity sense if and only if $\vp_t(x,t)-\cI u(x,t)$
is well-defined and
\begin{equation}\label{eq-viscosity-sub}\vp_t(x,t)-\cI u(x,t)\le
f(x,t)\,\,\text{ for any $(x,t)\in\Omega\times J$ and
$\vp\in\rC^2_{\Omega\times J}(u;x,t)^+$.}
\end{equation}

$(b)$ If $u:\BR^n\times I\to\BR$ is a function which is lower
semicontinuous on $\overline\Omega\times J$, then $u_t-\cI u\ge f$ on $\Omega\times J$
in the viscosity sense if and only if $\vp_t(x,t)-\cI u(x,t)$
is well-defined and
\begin{equation}\label{eq-viscosity-super}\vp_t(x,t)-\cI u(x,t)\ge
f(x,t)\,\,\text{ for any $(x,t)\in\Omega\times J$ and
$\vp\in\rC^2_{\Omega\times J}(u;x,t)^-$.}
\end{equation}

$(c)$ If $u:\BR^n\times I\to\BR$ is a function which is continuous
on $\overline\Om$, then $u$ is a viscosity solution to $u_t-\cI u=f$
on $\Omega\times J$ if and only if it satisfies both \eqref{eq-viscosity-sub}
and \eqref{eq-viscosity-super}.
\end{thm}

\pf Refer to \cite{CS, KL1}.\qed

The following  comparison principle and stability of viscosity solutions come from \cite{CS} with minor changes as \cite{W1}.

\begin{thm}\label{thm-comparison} Let $\Om$ be a bounded domain in $\BR^n$ and let $\cJ$ be elliptic with respect to
a class $\fL$ in Definition \ref{def-class}. If
$u\in\rB(\BR^n\times J)$ is a viscosity subsolution to $\cJ u-u_t\ge f$ on
$\Om\times J$, $v\in\rB(\BR^n\times J)$ is a viscosity supersolution to $\cJ v-v_t\le f$
on $\Om\times J$ and $u\le v$ on $(\BR^n \times (0,T]\s \Om\times J)\cup (\BR^n\times \{0\})$, then $u\le v$ on
$\Om\times J$.
\end{thm}

\begin{thm}\label{thm-2.0.4} Let $\Om$ be a bounded domain in $\BR^n$ and let $\cJ$ be elliptic with respect to a
class $\fL$ as in Definition \ref{def-class}. If
$u\in\rB(\BR^n\times I)$ is a viscosity subsolution to $\cJ u-u_t\ge
f$ on $\Omega\times J$ and $\,v\in\rB(\BR^n\times I)$ is a viscosity
supersolution to $\cJ v-v_t\le g$ on $\Omega\times J$ where $\Omega\times J=\Om\times
J\subset\BR^n\times I$, then $\cM^+_{\fL}[u-v]-(u_t-v_t)\ge f-g$ on
$\Omega\times J$ in the viscosity sense.
\end{thm}

\begin{thm}\label{thm-stability} Let $\cJ$ be elliptic as Definition \ref{def-class}.
If $\,\{u_k\}\subset\rB(\BR^n\times J)$ is a sequence of
viscosity solutions to $\cJ u_k-\partial_t u_k=f_k$ on $\Om\times J$ such that

$(a)$ $\{u_k\}$ and $\{f_k\}$ converge to $u$ and $f$ locally
uniformly on $\Om\times J$, respectively,

$(b)$ $\{u_k\}$ converges to $u$ $\aee$ on $\BR^n \times J$,

\noindent then $u$ is a viscosity solution to $\cJ u-\partial_t u=f$ on
$\Om\times J$.\end{thm}

\section{ A nonlocal Alexandroff-Bakelman-Pucci estimate}\label{sec-abp}
\subsection{$\vep $-envelope and Monge-Amp\'ere measure}
We employ the concept of $\vep$-envelope given at \cite{W1}.
\begin{definition}
$(i)$ Define the Minkowski sum $A+B$ of two sets $A$ and $B$ in
$\R^{n+1}$ as $A+B=\{x+y:\, x\in A, y\in B\}.$

$(ii)$ Set $\{0\}^{\vep}:=\{(x,t,z)\in\BR^n\times\BR\times\BR:\,
|x|^2-t+|z|^2\leq\vep^2\text{ for $t\leq 0$}\,\}$ and
$A^{\vep}:=A+\{0\}^{\vep}.$

$(iii)$ For $G(u):=\{(x,t,z)\in\BR^n\times\BR\times\BR:\, z\leq
u(x,t)\}$, we define the upper $\vep$-envelope of $u$ as
$u^{\vep}(x,t)=\sup\{u: (x,t,z)\in G(u)^{\vep}\}$ and lower
$\vep$-envelope of $u$ as $u_{\vep}=-(-u)^{\vep}$.
\end{definition}
Then it has the following properties.

\begin{lemma}[\cite{W1}] Let $u$ be a viscosity solution of $u_t=\cI[u]$ for
$\cI\in\fL$. Then we have the following properties;

$(i)$ $u^{\vep}\in C(D)$ and $u^{\vep}\downarrow u$ uniformly in $D$
as $\vep\rightarrow 0$.

$(ii)$ $u^{\vep}$ is a viscosity subsolution.

$(iii)$ For any $(x_0,t_0)\in D$, there is a concave parabolic
paraboloid of opening $2/\vep$ that touches $u^{\vep}$ at
$(x_0,t_0)$ from below in $D$. So $u^{\vep}$ is $C^{1,1}$ by below
in $D$. In particular, the lower bounds of $D^2_x u^{\vep}$ and
$-u^{\vep}_t$ are well defined in $D$ and $u^{\vep}\in
C_{x,-}^{1,1}(x,t)\cap C_t^{0,1}(x,t)\, a.e. $ in $D$.
\end{lemma}

\begin{definition}
Let $u$ be a convex function in $\Omega\subset \R^n$. Then the {\rm
Monge-Amph\`ere} measure $M_u$ corresponding to the function $u$ is
defined as
$$M_u(A)=\int_A\det (D^2 u)\,dx$$
for a subset $A\subset \Omega$.
\end{definition}

\begin{lemma}[Chapter VII,\cite{D}]
Let $u_k$ be convex functions satisfying $u_k\rightarrow u$
pointwise in a convex domain $\Omega$ and let $A_k\rightarrow A$ as
$k\rightarrow \infty$ where $A_k$'s and $A$ are bounded closed
subsets of $\Om$. Then we have that
$$\limsup_{k\rightarrow\infty}M_{u_k}(A_k)\leq M_u(A)$$ and
$$\lim_{k\rightarrow\infty}\int_{\Omega}f\,dM_{u_k}= \int_{\Omega}
f\,dM_{u}$$ for $f\in C^0_c(\Omega)$.
\end{lemma}
\subsection{Concave envelope and normal map}
We now define concave envelopes of a function $u$ defined on
$\BR^n\times I$ and furnish their properties below. For $r>0$, we
set $Q_r=Q_r(0,0)$.

\begin{definition}\label{def-gamma}  Let $u:\BR^n\times I\to\BR$ be a function which is
not positive on $\partial_p^* Q_{r/2}$ and is upper semicontinuous
on $\overline Q_r$.

$(i)$ $u(x,t)$ is called concave in $\BR^n\times I$ if $u(x,t)$ is
concave in $x$ and nondecreasing in $t$.

$(ii)$ The concave envelope $\Gm(y,s)$  of $u$ in $Q_{2r}$ is
defined as
$$\Gm(y,s)=\begin{cases}\inf\{v(y,s): v\in\Pi,\,v>u^+\text{ in
$Q_{2r}$}\}&\text{ in $Q_{2r}$ }\\
0 &\text{ in $\partial_p^* Q_{2r}$,}\end{cases}$$ where $\Pi$ is
the family of all concave functions $v$ in $Q_{2r}$ such that
$v\leq 0$ on $\partial_p Q_{2r}$.

$(iii)$ The {\it normal map} $\fN_u$ of $u:Q\rightarrow\BR$ is given
by
\begin{equation}
\begin{split}
\fN_u(x,t)=\{(p,h)\in \BR^n\times (\tau_1,t]:\,\, & u(y,s)\leq u(x,t)+p\cdot (y-x),\,\forall (y,s)\in Q,\\
&\qquad\,\,\text{and}\,\,u(x,t)-p\cdot x=h\}\end{split}
\end{equation}
\end{definition}

\begin{lemma}[Chapter VII,\cite{D}] Let $u:\BR^n\times I\to\BR$ be a function which is
not positive on $\partial_p^* Q_{r/2}$ and is upper semicontinuous on
$\overline Q_r$. Then we have the followings;

$(i)$ if $u\in C(\overline Q_r)$, then its concave envelope $\Gm$ is
nondecreasing in $t$ and $\Gm\in C(\overline Q_r)$.

$(ii)$ if $u\in Lip(\overline Q_r)$, then $\Gm\in Lip(\overline
Q_r)$ and $\|\Gm\|_{ Lip(\overline Q_r)}\leq (n+1)\|u\|_{
Lip(\overline Q_r)}$.
\end{lemma}

\begin{lemma}\label{lem-gamma-t}
Let $\Gamma$ be concave in $Q_1$ and $\Gamma=0$ on $\partial_p Q_1$.
Set $\cC$ to be the support of $\det (D^2 \Gamma)$. Then we have the
following results;

$(i)$ if $0\leq \Gamma_t\leq M$ on $\cC$, then $0\leq\Gamma_t\leq
2M$ on $\overline{Q}_1$.

$(ii)$ for any $(x,t)\in Q_1\cap\cC$, we have that $$0\leq
\sup_{B_r(x)}\Gamma_t(y,t) \leq \left((1-c r)\sup_{B_r(x,t)\cap
\cC}\Gamma_t(y,t)+\f{1}{c}\,r \sup_{Q_1}\Gamma_t(y,s)\right)\text{
on $\overline{Q}_1$}$$
for some uniform constant $c>0$.
\end{lemma}
\begin{proof}
(i) Take any point $(x_0,t_0)\in Q_1\s \cC$. Then there are some
$(x_i,t_0)\in \cC$ ($i=1,\cdots, n+1$) such that
$$\Gamma (x_0,t_0)=\sum_{i=1}^{n+1}\lambda_i \Gamma(x_i,t_0).$$
From the assumption, there is some $\delta\in (0,1)$ such that
$$\Gamma(x_i,t)\geq \Gamma(x_i,t_0)+2M (t-t_0)$$
for $t_0-\delta\leq t\leq t_0$. Set $\tilde{\Gamma}$ to be the
smallest concave function in $Q=\{(x,t)\in Q_R:\, t_0-\dt\leq t\leq
t_0\}$ such that $\tilde{\Gamma}=0$ on $\partial_p Q$ and

$$\tilde{\Gamma}(x_i,t)\geq \Gamma(x_i,t_0)+2M (t-t_0)$$
for $t_0-\delta\leq t\leq t_0$. Since $\Gamma(x_i,t)\geq
\Gamma(x_i,t_0)+2M (t-t_0)$, $\Gamma (x,t)\geq \tilde{\Gamma}(x,t)$
in $Q$ and $\Gamma (x_0,t_0)=\tilde{\Gamma}(x_0,t_0)$ from the
definition of  $\tilde{\Gamma}(x,t)$. Therefore
$\Gamma_t(x_0,t_0)\leq \tilde{\Gamma}_t(x_0,t_0)\leq 2M$. And from
the definition of concavity of $\Gamma$ in a parabolic domain and
Definition \ref{def-gamma}, we see that $\Gamma$ is nonincreasing. So
we have that $\Gamma_t\geq 0$.

(ii) From the argument in  (i), for any point $(x_0,t)\in
(B_r(x)\times \{t\})\s \cC$, there  are  $(x_i,t)\in \cC$ ($i=1,\cdots,
n+1$) such that $\Gamma (x_0,t)=\sum_{i=1}^{n+1}\lambda_i
\Gamma(x_i,t)$ and $\partial_t\Gamma
(x_0,t)=\sum_{i=1}^{n+1}\lambda_i\partial_t \Gamma(x_i,t).$ In
addition, there is some $x_k\in B_r(x)\cap\cC$ and $c\in(0,1)$ such
that
$$c r<1-\lambda_k<\f{r}{c}$$ from the convexity of $\Gamma$ since $x\in B_r(x)\cap \cC$. Then we
have that
\begin{equation*}
\begin{split}
\partial_t\Gamma (x_0,t)&\leq \lambda_k \partial_t \Gamma(x_k,t)+\sum_{i\neq k}\lambda_i \partial_t\Gamma(x_i,t)\\
&\leq \left((1-cr)\sup_{y\in B_r(x)\cap
\cC}\partial_t\Gamma(y,t)+\f{1}{c}\,r
\sup_{Q_1}\partial_t\Gamma\right)
\end{split}
\end{equation*}
since $\sum_i\lambda_i=1$.

\end{proof}

\begin{lemma}\label{lem-normal-map}
$(i)$ If $\,v$ is  strictly convex and smooth, then we have that
$\n_i\left(\det (D^2 v(x,t))v^{ij}\right)=0$ where $(v^{ij})$ is the
inverse of $(v_{ij})$.

$(ii)$ If $v\in C^{2,1}(\overline{Q}_r(x_0,t_0))$ is concave in $x$
and increasing in $t$ and if $v=0$ on $\partial_pQ_r(x_0,t_0)$, then
we have that
\begin{equation*}\begin{split}
&(a)\,\,\,\int_{B_r(x_0)}v(x,t_0)\det (D^2 v(x,t_0))dx
\\&\qquad\qquad\qquad\qquad\qquad\qquad=(n+1)\int_{Q_r(x_0,t_0)}\partial_t v(x,t)\det (D^2 v
(x,t))dxdt,\end{split}\end{equation*}

\item$(b)$ $$\fN_v (Q_r)=\int_{Q_r}\partial_t v(x,t)\det (D^2 v(x,t))^-dxdt,$$
\item$(c)$ $$\max_{x\in B_r}v(x,t_0)\leq C r\,\left(\fN_v (Q_r(x_0,t_0))\right)^{\frac{1}{n+1}}.$$
\end{lemma}
\begin{proof}
(i) Set $K=\det(D^2 v)$. By taking a derivative of
$K\delta^j_i=Kv^{jk}v_{ki}$ in the direction $e_j$, we have
$Ku^{mn}v_{mnj}\delta^j_i=(Kv^{jk})_jv_{ki}+Kv^{jk}v_{kij}$, which
implies $ (Kv^{jk})_jv_{ki}=0$. Now we show $ 0=
(Kv^{jk})_jv_{ki}v^{il}=(Kv^{jk})_j\delta_k^l=(Kv^{jl})_j$.

(ii-(a)) For readers, we are going to show (b) when $v$ is smooth and
strictly convex. The general case can be proved by approximation,
Theorem 22, \cite{D}.

\begin{equation*}
\begin{split}
\frac{d}{dt}\int_{B_r(x_0)}&v(x,t)\det (D^2
v(x,t)dx\\
&=\int_{B_r(x_0)}v_t\det (D^2 v)dx+\int_{B_r(x_0)}v\det (D^2
v)v^{ij}v_{t,ij}dx.
\end{split}
\end{equation*}
By taking an integration by part at $i,j$ variables and applying
(i), we can show
\begin{equation*}
\begin{split}
\int_{B_r(x_0)}v\det (D^2 v)v^{ij}v_{t,ij}=\int_{B_r(x_0)}v\det (D^2 v)v^{ij}v_{t,ij}\\
=\int_{B_r(x_0)}v_t\det (D^2 v)v^{ij}v_{ij}=n\int_{B_r(x_0)}v_t\det
(D^2 v)dx
\end{split}
\end{equation*}
and then
\begin{equation*}
\begin{split}
\frac{d}{dt}\int_{B_r(x_0)}v(x,t)\det (D^2
v(x,t)dx=(n+1)\int_{B_r(x_0)}v_t\det (D^2 v)dx.
\end{split}
\end{equation*}
Taking an integration in $t$ on $(t_0-r,t_0]$, we get the
conclusion.

(ii-(b)) It comes from \cite{T}.

(ii-(c)) The proof of (c) can be found at Theorem 22, Chapter
VII,\cite{D}. The idea is the following. First choose a concave cone
$\tilde{\Gamma}$ in $Q_{2r}(x_0,t_0)$  whose vertex is
$\max_{B_r}v(x,t_0)$ and  $\supp\tilde{\Gamma}=Q_{2r}(x_0,t_0)$.
Then $\fN_ {\tilde{\Gamma}}(Q_r(x_0,t_0))\leq \fN_v (Q_r(x_0,t_0)).$
Since $\supp \det (D^2 \tilde{\Gamma})$ is the maximum point, we
have
\begin{equation*}
\begin{split}
\left(\frac{\max_{B_r}v(x,t_0)}{r}\right)^{n+1}&\leq C(n+1)\int_{B_r(x_0)}(\tilde{\Gamma}(x,t_0))\det (D^2 \tilde{\Gamma}(x,t_0)dx\\
&= C(n+1)\,\fN_{\tilde{\Gamma}} (Q_r(x_0,t_0))  .
\end{split}
\end{equation*}
\end{proof}

\subsection{Nonlocal Parabolic A-B-P estimate}

\begin{lemma}\label{lem-abp-k} Let $0<\sigma<2$ be given.
Let $u\le 0$ in $\partial_p^* Q_{1/2}$ and let $\Gm$ be its
concave envelope in $Q_{2}$. If $u\in \rB(\BR^n\times I)$ is a
viscosity subsolution to $\cM^+_{\fL} u-u_t=-f$ in $Q_1$ where
$f:\BR^n\times I\to\BR$ is a continuous function with $f>0$ on
$\cC(u,\Gm,Q_1)$, then there exists some constant $C>0$ depending
only on $n,\ld$ and $\Ld$ $($but not on $\sm$$)$ such that for any
$(x,t)\in\cC(u,\Gm,Q_1)$ and any $\vr,M>0$ there is some
$k\in\BN\cup\{0\}$ such that
\begin{equation}\label{eq-abp-1}
\bigl|\{y\in R_k:u(x+y,t)-u(x,t)-(\n_x\vp(x,t)\cdot y)\chi_{B_1}(y)\geq  M\,r_k^2\}\bigr|\le
C\,\f{f(x,t)}{M}\,|R_k|
\end{equation} where $R_k=B_{r_k}\s
B_{r_{k+1}}$ for $r_k=\vr\,2^{-\f{1}{2-\sm}-k}$. In addition, we
have that \begin{equation}\partial_t\Gamma (x,t)\leq C
f(x,t).\end{equation} Here, $\n\Gm(x,t)$  and $\partial_t\Gm(x,t)$ denote any element of the
superdifferential $\pa\Gm(x,t)$ of $\,\Gm$ with respect to $x$ and $t$ respectively at
$(x,t)$.
\end{lemma}

{\it $[$Proof of Lemma \ref{lem-abp-k}$]$} Let $0<\sigma<2$ be
given. Take any $(x,t)\in\cC(u,\Gm,Q_1)$. Since $u$ can be touched
by a hyperplane from above at $(x,t)$, we see that
$\n\vp(x,t)=\n\Gm(x,t)$ for some $\vp\in \rC^2_{Q_1}(u;x,t)^+$.

If $|y|<1$, then
$\mu(u,x,y,t)\le 0$ by the
definition of $\Gm$. If $|y|\ge 1$, then $|x\pm y|\ge |y|-|x|\ge
1-1/2=1/2$ because $\cC(u,\Gm,Q_1)\subset Q_{1/2}$. Thus we have
that $u(x\pm y,t)\le 0$, and so we see that $\mu(u,x,y,t)\le 0$ for
any $y$ with $|y|\ge 1$. Therefore we conclude that
$\mu(u,x,\cdot,t)\le 0$ on $\BR^n$. This implies that
$\mu^+(u,x,y,t)=0$. Since $\Gamma_t\geq 0$,  we have that
\begin{equation}
\begin{split}
-f(x,t)&\le\cM^+_{\fL} u(x,t)-\partial_t\Gm(x,t)\\
&=(2-\sm)\int_{\BR^n}\f{-\ld\mu^-(u,x,y,t)}{|y|^{n+\sm}}\,dy-\partial_t\Gm(x,t)\\
&\le(2-\sm)\int_{B_{r_0}}\f{-\ld\mu^-(u,x,y,t)}{|y|^{n+\sm}}\,dy
\end{split}
\end{equation} where $r_0=\vr\,
2^{-1/(2-\sm)}$. Decomposing the above integral into the rings
$R_k$, we have that
\begin{equation}\label{eq-abp-2}
f(x,t)\ge(2-\sm)\ld\sum_{k=0}^{\iy}\int_{R_k}\f{\mu^-(u,x,y,t)}{|y|^{n+\sm}}\,dy.
\end{equation}

Assume that the conclusion \eqref{eq-abp-1} does not hold, i.e. for
any $C>0$ there are some $(x_0,t_0)\in\cC(u,\Gm,Q_1)$ and
$\vr_0,M_0>0$ such that
\begin{equation*}\label{eq-abp-3}
\bigl|\{y\in R_k: \mu^-(u,x_0,y,t_0)\geq
M\,r_k^2\}\bigr|>C\,\f{f(x_0,t_0)}{M_0}\,|R_k|
\end{equation*}
for all $k\in\BN\cup\{0\}$. Since $(2-\sm)\frac{1}{1-2^{-(2-\sm)}}$
remains bounded below for $\sm\in (0,2)$, it thus follows from
\eqref{eq-abp-2} that
\begin{equation}
\begin{split}
f(x_0,t_0)&\ge(2-\sm)\ld\sum_{k=0}^{\iy}\int_{R_k}\f{\mu^-(u,x_0,y,t_0)}{|y|^{n+\sm}}\,dy\\
&\ge
c(2-\sm)\sum_{k=0}^{\iy}M_0\,\f{r_k^2}{r_k^{\sm}}\,C\,\f{f(x_0,t_0)}{M_0}\\
&\ge\f{c(2-\sm)\rho_0^2}{1-2^{-(2-\sm)}}\,C\,f(x_0,t_0)\ge
c\,C\,f(x_0,t_0)
\end{split}\end{equation}
for any $C>0$. Taking $C$ large enough, we obtain a contradiction.
Hence we are done.

Now we are going to control the time derivative of $\Gamma$. Since
$\mu(u,x,y,t)\leq 0$ for $y\in \BR^n$ as in the above, we have that
\begin{equation*}
\partial_t\Gm(x,t)\le\cM^+_{\fL} u(x,t)+f(x,t)\le f(x,t)
\end{equation*} Hence we complete the proof. \qed

Lemma \ref{lem-abp-k} and Lemma \ref{lem-gamma-t} give us the
following corollary.
\begin{cor}\label{cor-abp-t} Under the same condition as Lemma \ref{lem-abp-k},
there is a universal constant $C>0$ such that
$$\sup_{Q_1}\partial_t\Gamma \leq C\,\sup_{Q_1} f.$$
\end{cor}

To get a local estimate, let us introduce the following rectangles which is different from the standard parabolic rectangle since it incluses some future points:
for $(x,t)\in Q_1$, we set $$\cK_r(x,t)=\{(y,s):\, \max(
|y_1-x_1|^2,\cdots,|y_n-x_n|^2, |t-s|)< r^2\}.$$
\begin{lemma}\cite{CS}\label{lem-abp-2} Let $\Gm$ be a parabolic concave function on $K_r(x,t)$
where $x\in\BR^n$ and let $h>0$. If $|\{y\in
S_r(x):\Gm(y,t)<\Gm(x,t)+(y-x)\cdot\n_x\Gm(x,t)-Cf(x,t)r^2\}|\le\ep\,|S_r(x)|$
for any small $\ep>0$ where $S_r(x)=B_r(x)\s B_{r/2}(x)$, then we
have $\Gm(y,t)\ge\Gm(x,t)+(y-x)\cdot\n_x\Gm(x,t)- Cf(x,t)r^2$ for
any $(y,t)\in K_{r/2}(x,t)$.\end{lemma}

\begin{cor}\label{cor-abp-q}
Under the same condition as Lemma \ref{lem-abp-2}, there is some
universal constant $C>0$ such that
 $$\left|\Gm(y,s)-\Gm(x,t)+(y-x)\cdot\n_x\Gm(x,t)\right|\leq C\,r^2\biggl(\,\sup_{\cK_{r/2}(x,t)}f+r\sup_{Q_1}f\,\biggr)$$ for any
$(y,s)\in \cK_{r/2}(x,t)$.
\end{cor}
\begin{proof} We observe that $\pa_t\Gm(y,s)\le C\,f(y,s)$ for any
$(y,s)\in \cK_{r/2}(x,t)\cap\cC(u,\Gm,Q_1)$. 
(1.) First consider the case $s\leq t$. From Corollary
\ref{cor-abp-t} and (ii) of Lemma \ref{lem-gamma-t}, we have that
\begin{equation*}\begin{split}
\Gm(y,s)&\geq \Gm(y,t)+(s-t)\partial_t\Gm(y,t)-o(|s-t|)\\
&\ge\Gm(x,t)+(y-x)\cdot\n_x\Gm(x,t)-Cf(x,t)r^2\\
&\quad+(s-t)\bigl((1-c\,r)\sup_{y\in
B_{r/2}(x)\cap\cC}\pa_t\Gm(y,t)+c\, r\sup_{Q_1}\pa_t\Gm\bigr)\\
&\ge\Gm(x,t)+(y-x)\cdot\n_x\Gm(x,t)-C\,r^2\biggl(\,\sup_{\cK_{r/2}(x,t)}f+r\sup_{Q_1}f\,\biggr)
\end{split}\end{equation*} for any $(y,s)\in
K_{r/2}(x,t)$ such that $s\leq t$. 
And 
\begin{equation*}\begin{split}
\Gm(y,s)&\leq \Gm(y,t)\\
&\le\Gm(x,t)+(y-x)\cdot\n_x\Gm(x,t)
\end{split}\end{equation*}
for the concavity of $\Gamma$.

(2.) Now we assume $s>t$.
\begin{equation*}\begin{split}
\Gm(y,s)&\geq \Gm(y,t)\\
&\ge\Gm(x,t)+(y-x)\cdot\n_x\Gm(x,t)-Cf(x,t)r^2\\
\end{split}\end{equation*} for any $(y,s)\in
\cK_{r/2}(x,t)$. 
And \begin{equation*}\begin{split}
\Gm(y,s)&\leq \Gm(y,t)+(s-t)\partial_t\Gm(y,t)+o(|s-t|)\\
&\le\Gm(x,t)+(y-x)\cdot\n_x\Gm(x,t)\\
&\quad+(s-t)\bigl((1-c\,r)\sup_{y\in
B_{r/2}(x)\cap\cC}\pa_t\Gm(y,t)+c\, r\sup_{Q_1}\pa_t\Gm\bigr)\\
&\le\Gm(x,t)+(y-x)\cdot\n_x\Gm(x,t)+C\,r^2\biggl(\,\sup_{\cK_{r/2}(x,t)}f+r\sup_{Q_1}f\,\biggr)
\end{split}\end{equation*} for any $(y,s)\in
\cK_{r/2}(x,t)$ such that $s>t$. 
\end{proof}

\begin{cor}\label{cor-abp-3}
For any $\ep>0$, there is a constant $C>0$ such that for any
function $u$ with the same hypothesis as Lemma \ref{lem-abp-k} and
for each $(x,t)\in\cC(u,\Gm,Q_1)$, there is some $r>0$ and
$r\in(0,2\rho_0 2^{-\f{1}{2-\sm}})$ such that
\begin{equation*}\begin{split} &\f{|\{y\in S_r(x):u(y,t)<u(x,t)+(y-x)\cdot\n_x\Gm(x,t)-C\xi_r(x,t)\}|}{|S_r(x)|}\le\ep,\\
&|p(y,s)-\Gm(y,s)|\le C\xi_r(x,t)
\end{split}\end{equation*} for any $(y,s)\in \cK_{r/2}(x,t)$, where
$p(y,s)=\Gm(x,t)+(y-x)\cdot\n_x\Gm(x,t)$ and
$\xi_r(x,t)=\,r^2\left(\sup_{\cK_{r/2}(x,t)}f+r\sup_{Q_1}
f\right)$.\end{cor}

Using Tso's argument on \cite{D} and \cite{T}, we easily obtain the
following lemma.

\begin{lemma}\label{lem-tso} There is a universal constant $C_n>0$ such that for any
function $u$ with the same hypothesis as Lemma \ref{lem-abp-k},
\begin{equation*}\bigl(\sup_{Q_1}u^+\bigr)^{n+1}\le C_n\,|\fN_{\Gm}(Q_1)|.\end{equation*}\end{lemma}

\begin{cor}\label{cor-abp-n-upper}
For any $\ep>0$, there is a universal constant $C>0$ such that for
any function $u$ with the same hypothesis as Lemma \ref{lem-abp-k}
and for each $(x,t)\in\cC(u,\Gm,Q_1)$, there is a
$r\in(0,2\rho_0 2^{-\f{1}{2-\sm}})$ such that
\begin{equation*}\bigl|\fN_{\Gm}(\cK_{r/2}(x,t))\bigr|\le C\,\biggl(\,\sup_{\cK_{r/2}(x,t)}f+r\sup_{Q_1} f\,\biggr)^{n+1}\,|\cK_{r/2}(x,t)
\cap\cC(u,\Gm,Q_1)|.\end{equation*} \end{cor}

\pf  Fix any $(x,t)\in\cC(u,\Gm,Q_1)$. $\partial_t\Gamma (y,s)\leq
Cf(y,s)$ for $(y,s)\in \cK_{r/2}(x,t) \cap\cC(u,\Gm,Q_1)$. And from
Corollary \ref{cor-abp-3} and  Lemma \ref{lem-normal-map}, we
conclude that
\begin{equation*}\begin{split}|\fN_{\Gm}(\cK_{r/2}(x,t))|&\le C\int_{\cK_{r/2}(x,t)
\cap\cC(u,\Gm,Q_1)}\pa_t\Gm(y,s)\det[D^2_x\Gm(y,s)]^-\,dy\,ds\\
&\le C\,\biggl(\,\sup_{\cK_{r/2}(x,t)}f+r\sup_{Q_1}
f\,\biggr)^{n+1}\,|\cK_{r/2}(x,t)
\cap\cC(u,\Gm,Q_1)|.\end{split}\end{equation*} Hence we complete the
proof. \qed

We obtain a nonlocal version of Alexandroff-Bakelman-Pucci estimate
in the following theorem as

\begin{thm}\label{thm-abp} Let $u$ and $\Gm$ be functions as in
Lemma \ref{lem-abp-k}. Then there exist a finite family 
$\{K_k\}$ of pairwise disjoint $(n+1)$-dimensional
cubes with sidelength $r_k/2$ such that

$(a)$ $\cC(u,\Gm,Q_1)\subset\bigcup_k K_k$,

$(b)$ $\cC(u,\Gm,Q_1)\cap K_k\neq\phi$ for any $k$,
$(c)$ $r_k\le 2\rho_0 2^{-\f{1}{2-\sm}}$ for any $k$,

$(d)$ $|\fN_{\Gm}(K_k)|\le C\,\bigl(\sup_{\overline
K_k}f+r_k\sup_{Q_1} f\bigr)^{n+1}$

$\qquad\qquad\qquad\qquad\qquad\quad\qquad\qquad\qquad\times|\cC(u,\Gm,Q_1)\cap
K_k|$

$(e)$ $\bigl(\sup_{Q_1}u^+\bigr)^{n+1}\le C\,|\fN_{\Gm}(Q_1)|$

\noindent where the constants $C>0$ depends on $n,\Ld$ and $\ld$ $($
but not on $\sm$$)$.\end{thm}

\pf First we make dyadic disjoint $(n+1)$-dimensional cubes as the elliptic case, \cite{CS, KL1,KL2}. It follows from Lemma \ref{lem-tso}, Corollary
\ref{cor-abp-n-upper} with the same reason as \cite{CS, KL1,KL2}. Otherwise, there is a sequence of contact points, $\{(x_j,t_j)\}$ which belong to  dyadic cubes, $K_{k_j}$  with lengths converging to zero i.e. $\{(x_j,t_j)\} \subset Q_1\cap\cC$, $(x_j,t_j)\in K_{k_j}$  and $r_{k_j}\rightarrow 0$. Since $Q_1$ is compact, a subsequence of  $\{(x_j,t_j)\}$ converges to a point $(x_0,t_0)\in Q_1\cap\cC$. We will use the same notation for the subsequence of $\{(x_j,t_j)\}$ . On the other hand there is a cube $\cK_{r/2}(x_0,t_0)$ satisfying the conditions (a,b,c) and there is a large $N>0$ such that $Q_{k_j}\subset \cK_{r/2}(x_0,t_0)$ for $i>N$, which is a contradiction.
\qed

\rk It follows from Theorem \ref{thm-abp} we have that
\begin{equation*}\begin{split}\bigl(\sup_{Q_1}u^+\bigr)^{n+1}&\le\sum_k|\fN_{\Gm}(K_k)|\\
& \le C\sum_k\biggl(\,\sup_{\overline
K_k}f+r_k\sup_{Q_1} f\biggr)^{n+1}|\cC(u,\Gm,Q_1)\cap
K_k|.\end{split}\end{equation*} As $\sm\to 2$, the
cube covering of $\cC(u,\Gm,Q_1)$ is getting close to the contact
set $\cC(u,\Gm,Q_1)$ and so the above becomes the following estimate
$$\sup_{Q_1}u^+\le
C\biggl(\int_{\cC(u,\Gm,Q_1)}[f(y,s)]^{n+1}\,dy\,ds\biggr)^{1/(n+1)}.$$
Our estimates remain uniform as the index $\sm$ of the operator is
getting close to $2$. Therefore this implies that our A-B-P estimate
can be regarded as a natural extension of that for parabolic partial
differential equations.

\section{Decay Estimate of Upper Level Sets}\label{sec-decay}

In this section, we are going to show the geometric decay rate of
the upper level set of nonnegative solution $u$. The key Lemma
\ref{lem-decay-key} says that if  a nonnegative function $u$ has a
value smaller than one in $\rK^+_{3 r_0}$ then the lower level set
$\{u\leq M\}\cap\rK^-_{r_0}$ has at least uniformly positive amount of
measure $\nu |\rK^-_{r_0}|$ which will be proven through ABP
estimate. But the assumption of ABP estimate on a subsolution
requires its special shape: it should be nonpositive  $\partial^*_p Q$ and
positive at some interior point. So we are going to construct a
special function $\Psi$ so that $\Psi-u$ meets the requirement of
ABP estimate.

\subsection{Special functions}

The construction of the special function is based on the idea in
\cite{CS, KL1, KL2, W1}. Nontrivial finer computation has been done to detect the influence of values along the time.

\begin{lemma}\label{lem-barrier-1}
Let $Q=B_1\times (0,1]$ and $0<r<1/(2\sqrt{n})$. Then there exist
some $\sm^*\in (1,2)$  and a subsolution $\psi(x,t)$ such that for
any $\sm\in (\sm^*,2)$,
\begin{equation}
\begin{cases}
\psi_t(x,t)\leq \cM^-_{\fL} \psi(x,t),\quad (x,t)\in  Q\s \rK^-_{r/2}\\
\psi(x,t)\geq 1\quad\text{in $\rK^+_{3r}$}\\
\psi(x,t)\leq 0\quad\text{on $\partial^*_p Q$}.
\end{cases}
\end{equation}
\end{lemma}

\pf We set $$h(x,t)=\min\left\{2^n,\frac{1}{(4\pi
t)^{n/\sigma}}\exp\left(-\alpha\frac{|x|^\sigma}{t}\right)\right\},\,\,x\in\BR^n,t\in
I,$$ and $f(x,t)=e^{-\beta t}h(x,t)^{\gamma}$ for
$\ap,\bt,\gm>0$. If $|x|^{\sigma}\ge t$ and $0<t\leq
(2^{-\sigma}e^{-\alpha\sigma/n})/4\pi$, then we consider the
function $g(y,s)=|x|^n f(x,t)$ where $x=|x|y$ and $t=|x|^{\sm}s$.
Then $g(y,s)=e^{-\beta \eta^{\sigma}s}h(y,s)^{\gamma}$ for $\eta=|x|$ and we note that
$$\cL f(x,t)=\ds\f{1}{|x|^{n+\sm}}\cL
g\biggl(\f{x}{|x|},\f{t}{|x|^{\sm}}\biggr)\,\,\text{ and
}\,\,f_t(x,t)=\f{1}{|x|^{n+\sm}}g_t\biggl(\f{x}{|x|},\f{t}{|x|^{\sm}}\biggr).$$
If we choose the normalization $y\in S^{n-1}$, then it is
enough to show that there is some $\sm^*\in(1,2)$ so that
\begin{equation}\label{eq-7.1}
\cM^-_{\fL}g(e_n,s)\ge g_s(e_n,s)\,\,\text{ on $Q\s
\rK^-_{r/2}$}
\end{equation}
for $x=e_n=(0,0,\cdots,0,1)\in\BR^n$; for, the above inequality
follows by scaling and rotation for every other $x$ with $|x|=1$.
Then we have that, $\tilde{\beta}=\beta \eta^{\sm}$
\begin{equation}\label{eq-7.2}\begin{split}
g_s= e^{-\tilde{\beta}
t}h^{\gamma}\biggl(-\tilde{\beta}+\gamma\alpha\frac{|x|^\sigma}{s^2}-\frac{\gamma
n}{\sigma s}\biggr)\,,
\,\,g_t(e_n,s)=\f{e^{-\tilde{\beta} t}e^{-\gm\tilde{\ap}}}{(4\pi s)^{n\gm/\sm}}\biggl(-\tilde{\beta}+\frac{\gm\tilde{\ap}}{s^2}-\f{\gm
n}{\sm s}\biggr)\,.\end{split}\end{equation} By the Taylor's
expansion, we obtain that, for $\tilde{\alpha}=\frac{\alpha}{t}$,
\begin{equation}\label{eq-7.3}\begin{split}
&e^{-\tilde{\ap}\gm|e_n\pm y|^{\sm}}=e^{-\tilde{\ap}\gm}\mp\sm\tilde{\ap}\gm
e^{-\tilde{\ap}\gm}y_n+\f{1}{2}\tilde{\ap}\gm\sm\,e^{-\tilde{\ap}\gm}\bigl\{\bigl(\tilde{\ap}\gm\sm+(2-\sm)\bigr)y_n^2-|y|^2\bigr\}\\
&\quad\mp\f{1}{6}\tilde{\ap}\gm\sm\,e^{-\tilde{\ap}\gm}\biggl\{\bigl[\tilde{\ap}\gm\sm\bigl(\tilde{\ap}\gm\sm+3(2-\sm)\bigr)+(2-\sm)(4-\sm)\bigr]y_n^3\\
&\quad\qquad\qquad\qquad\qquad\qquad\qquad\qquad\quad-3\bigl(\tilde{\ap}\gm\sm+(2-\sm)\bigr)y_n|y|^2\biggr\}\\
&\quad+\f{1}{24}\tilde{\ap}\gm\sm(2-\sm)\,e^{-\tilde{\ap}\gm}\biggl\{2\tilde{\ap}\gm\sm\bigl(3\tilde{\ap}\gm\sm+2(4-\sm)\bigr)+(4-\sm)(6-\sm)\biggr\}y_n^4\\
&\quad+\f{1}{4}\tilde{\ap}\gm\sm(2-\sm)\,e^{-\tilde{\ap}\gm}\biggl\{\tilde{\ap}\gm\sm\bigl(\tilde{\ap}\gm\sm+2(2-\sm)\bigr)+(2-\sm)(4-\sm)\biggr\}y_n^2|y|^2\\
&\quad+\f{1}{24}\tilde{\ap}^4\gm^4\sm^4\,e^{-\tilde{\ap}\gm}y_n^4+\f{1}{8}\tilde{\ap}\gm\sm(2-\sm)\,e^{-\tilde{\ap}\gm}|y|^4+o(|y|^4).
\end{split}\end{equation}
If we denote by $\sm_k=\int_{S^{n-1}}\theta_n^k\,d\sm(\theta)$ for
$k\in\BN$, then we may now choose some large enough $\gm>1$ so that
\begin{equation}\label{eq-7.4}
I(\sm_2,\tilde{\ap},\gm,\sm)\fd\bigl(\tilde{\ap}\gm\sm+(2-\sm)\bigr)\sm_2-\om_n>0
\end{equation} for any
$\sm\in(0,2)$, where $\om_n$ denote the surface measure of $S^{n-1}$
and $r>0$ is the constant to be given just below. Since the $4^{th}$
order term of \eqref{eq-7.3} is positive, by \eqref{eq-7.4} there is
a sufficiently small $r\in(0,1/2)$ (and fix) so that
\begin{equation}\label{eq-7.5}\begin{split}&(4\pi t)^{n\gm/\sm}\mu(g,e_n,y,s)\ge\tilde{\ap}\gm\sm\,
e^{-\tilde{\beta}s-\tilde{\ap}\gm}\bigl\{\bigl(\tilde{\ap}\gm\sm+(2-\sm)\bigr)y_n^2-|y|^2\bigr\}\\
&\quad+\f{1}{24}\tilde{\ap}\gm\sm(2-\sm)\,e^{-\tilde{\beta}s-\tilde{\ap}\gm}\biggl\{2\tilde{\ap}\gm\sm\bigl(3\tilde{\ap}\gm\sm+2(4-\sm)\bigr)+(4-\sm)(6-\sm)\biggr\}y_n^4\\
&\quad+\f{1}{4}\tilde{\ap}\gm\sm(2-\sm)\,e^{-\tilde{\beta}s-\tilde{\ap}\gm}\biggl\{\tilde{\ap}\gm\sm\bigl(\tilde{\ap}\gm\sm+2(2-\sm)\bigr)+(2-\sm)(4-\sm)\biggr\}y_n^2|y|^2\\
&\quad+\f{1}{24}\tilde{\ap}^4\gm^4\sm^4\,e^{-\tilde{\beta}s-\tilde{\ap}\gm}y_n^4+\f{1}{8}\tilde{\ap}\gm\sm(2-\sm)\,e^{-\tilde{\beta}s-\tilde{\ap}\gm}|y|^4>0
\end{split}\end{equation} for any $y\in B_r$. Then for a small $\tau>0$ which will be chosen later,  we have the
estimate,  by \eqref{eq-7.5},

\begin{equation}\label{eq-7.6}\begin{split}
&(2-\sm)\ld(4\pi t)^{n\gm/\sm}\int_{B_r\backslash B_{\tau} }\f{\mu(g,e_n,y,t)}{|y|^{n+\sm}}\,dy
\ge\f{\ld\tilde{\ap}\gm\sm}{e^{\tilde{\beta}s+\tilde{\ap}\gm}}\,I(\sm_2,\tilde{\ap},\gm,\sm)\,(r^{2-\sm}-\tau^{2-\sm})\\
&\quad+\f{\ld}{24}\f{\sm_4}{e^{\tilde{\beta}s+\tilde{\ap}\gm}}\biggl[\tilde{\ap}\gm\sm(2-\sm)\biggl\{2\tilde{\ap}\gm\sm\bigl(3\tilde{\ap}\gm\sm+2(4-\sm)\bigr)+(4-\sm)(6-\sm)\biggr\}\\
&\quad\qquad\qquad\qquad\qquad\qquad\qquad\qquad\qquad\quad+\tilde{\ap}^4\gm^4\sm^4\biggr]\,\f{2-\sm}{4-\sm}\,(r^{4-\sm}-\tau^{4-\sm})\\
&\quad+\f{\ld}{4}\f{\sm_2\tilde{\ap}\gm\sm}{e^{\tilde{\beta}s+\tilde{\ap}\gm}}\biggl\{\tilde{\ap}\gm\sm\bigl(\tilde{\ap}\gm\sm+2(2-\sm)\bigr)+(2-\sm)(4-\sm)\biggr\}
\f{(2-\sm)^2}{4-\sm}\,(r^{4-\sm}-\tau^{4-\sm})\\
&\quad+\f{\ld\om_n}{8}\f{\tilde{\ap}\gm\sm}{e^{\tilde{\beta}s+\tilde{\ap}\gm}}\,\f{(2-\sm)^2}{4-\sm}(r^{4-\sm}-\tau^{4-\sm})\\
&\quad\fd\f{\ld\tilde{\ap}\gm\sm}{2e^{\tilde{\beta}s+\tilde{\ap}\gm}}\,I(\sm_2,\tilde{\ap},\gm,\sm)\,r^{2-\sm}
+e^{-\tilde{\beta}-\tilde{\alpha}\gamma}a(\ld,\sm_2,\sm_4,\tilde{\ap},\tilde{\beta},\gm,\sm)\,\f{2-\sm}{4-\sm}\,r^{4-\sm}
\end{split}\end{equation} where $$\lim_{\sm\to
2^-}a(\ld,\sm_2,\sm_4,\tilde{\ap},\tilde{\beta},\gm,\sm)=\f{\ld}{24}\,\sm_4(8+16\tilde{\ap}^4\gm^4)>0.$$
Since $\mu^-(g,e_n,y,s)\le \frac{2^{n+2}}{(4\pi s)^{n \gamma/\sigma}} e^{-\tilde{\beta}s-\tilde{\alpha}\gamma}$ on $\BR^n$, it follows
from \eqref{eq-7.2}, \eqref{eq-7.4}, \eqref{eq-7.5} and
\eqref{eq-7.6} that
\begin{equation*}\begin{split}
&\cM_{\fL}\,g(e_n,s)-g_t(e_n,s)\\
&\ge(2-\sm)\ld\int_{\BR^n}\f{\mu^+(g,e_n,y,s)}{|y|^{n+\sm}}\,dy-(2-\sm)\Ld\int_{\BR^n}\f{\mu^-(g,e_n,y,s)}{|y|^{n+\sm}}\,dy-g_t(e_n,t)\\
&\ge(2-\sm)\ld\int_{B_r}\f{\mu(g,e_n,y,s)}{|y|^{n+\sm}}\,dy-(2-\sm)\Ld\int_{\BR^n\s
B_r}\f{\mu^-(g,e_n,y,s)}{|y|^{n+\sm}}\,dy-g_t(e_n,t)\\
&\ge \frac{e^{-\tilde{\beta}s-\tilde{\ap}\gm}}{(4\pi s)^{n\gm/\sm}}\left[\ld\tilde{\ap}\gm\sm\,I(\sm_2,\tilde{\ap},\gm,\sm)\,r^{2-\sm}
+a(\ld,\sm_2,\sm_4,\tilde{\ap},\tilde{\beta},\gm,\sm)\,\f{2-\sm}{4-\sm}\,r^{4-\sm}\right.\\
&\qquad\qquad\qquad\qquad\qquad\quad\left. -2^{n+2}\Ld\om_n\f{2-\sm}{\sm}r^{-\sm}\right]-g_t(e_n,s).
\end{split}\end{equation*}
Thus we may take some $\sm^*\in(1,2)$ close enough to $2$ in the
above and some sufficiently small $\tilde{\beta}\in(0,1)$ with $\tilde{\ap}=n/\sm$ so
that $\cM_{\fL}\,g(e_n,s)-g_t(e_n,s)\ge 0$ for any
$\sm\in(\sm^*,2)$. To complete the proof, we take
$\psi(x,t)=\min(\max(f(x,t)-\zeta,0), At)$ with a small $\zeta>0$ such that 
$\supp [\max(f(x,t)-\zeta,0)]\subset Q$ and a large $A>0$ such that
$\{(x,t): \,\max(f(x,t)-\zeta,0)\geq At\}\subset B_{\tau}\times [0,\tau^{\sigma}]$ for $0<\tau<<r$.
Now $\psi(x,t)$ which satisfies (4.1.1).  \qed

\begin{cor}\label{cor-barrier-2}
Let $Q=B_1\times (0,1]$ and $0<r<1/(2\sqrt{n})$. Given
$\sm_0\in(0,2)$, there is some very small $\dt\in(0,1)$ and a
subsolution $\psi^{\dt}(x,t)$ such that for any $\sm\in (\sm_0,2)$,
\begin{equation}
\begin{cases}
\psi^{\dt}_t(x,t)\leq \cM^-_{\fL} \psi^{\dt}(x,t),\quad (x,t)\in Q\s \rK^-_{r/2}\\
\psi^{\dt}(x,t)\geq 1\quad\text{in $\rK^+_{3r}$}\\
\psi^{\dt}(x,t)\leq 0\quad\text{on $\partial^*_p Q$}.
\end{cases}
\end{equation}\end{cor}

\pf Let $\sm^*\in (1,2)$ be the number of Lemma \ref{lem-barrier-1}.
Without loss of generality, we may assume that $\sm_0<\sm^*$. Lemma
\ref{lem-barrier-1} implies that the result of our corollary always
holds for $\sm\in(\sm^*,2)$, when $\dt=1/2$. We set
$$h_{\dt}(x,t)=\min\left\{\dt^{-n},\frac{1}{(4\pi
t)^{n/\sigma}}\exp\left(-\alpha\frac{|x|^\sigma}{t}\right)\right\},\,\,x\in\BR^n,t\in 
I,$$ and $f_{\dt}(x,t)=e^{-\tilde{\beta}
t}h_{\dt}(x,t)^{\gamma}$ for $\ap,\bt,\gm>0$. If
$|x|^{\sigma}\ge t$ and $0<t\leq
(\dt^{\sigma}e^{-\alpha\sigma/n})/4\pi$, then we consider the
function $g_{\dt}(y,s)=|x|^n f_{\dt}(x,t)$ where $x=|x|y$ and
$t=|x|^{\sm}s$. Set $\mu=|x|$. Then $g_{\dt}(y,s)=e^{-\beta \mu^{\sm} s
}h_{\dt}(y,s)^{\gamma}$ and we note that
$$\cL f_{\dt}(x,t)=\ds\f{1}{|x|^{n+\sm}}\cL
g_{\dt}\biggl(\f{x}{|x|},\f{t}{|x|^{\sm}}\biggr)\,\,\text{ and
}\,\,f_{\dt, t}(x,t)=\f{1}{|x|^{n+\sm}}g_{\dt,t}\biggl(\f{x}{|x|},\f{t}{|x|^{\sm}}\biggr).$$
If $\dt<1/2$, then the result still holds for $\sm\in(\sm^*,2)$
because $\mu(f_{\dt},x,y,t)\ge\mu(f_{1/2},x,y,t)$ for any
$x\in\BR^n$ and $t\in I$.

Now we let $x=e_n$ as in the proof of Lemma
\ref{lem-barrier-1}. Assume that $\sm_0<\sm\le\sm^*$. Then we write
\begin{equation*}
\begin{split}
\cM^-_{\fL}\,g_{\dt}(e_n,1)&=(2-\sm)\ld\int_{\BR^n}\f{\mu^+(g_{\dt},e_n,y,s)}{|y|^{n+\sm}}\,dy
-(2-\sm)\Ld\int_{\BR^n}\f{\mu^-(g_{\dt},e_n,y,s)}{|y|^{n+\sm}}\,dy\\
&\fd\cJ_1 g_{\dt}(e_n,s)+\cJ_2
g_{\dt}(e_n,s).
\end{split}
\end{equation*}
 If we take some $h>0$ and
$\dt\in(0,1)$ small enough so that $\mu^-(g_{\dt},e_n,y,s)=0$ for
any $y\in B_{1+h}$, from \eqref{eq-7.3} and simple geometric
observation it is easy to check that there is some $c>0$ not
depending on $\ap,\bt,\gm$ and $\sm$ such that, for $\tilde{\ap}=\frac{\ap}{t}$,
 $$\mu^-(g_{\dt},e_n,y,s)\le c\,\f{2 e^{-\bt\mu^{\sm} s-\tilde{\ap}\gm}}{(4\pi s)^{\gm n/\sm}}$$ for any $y\in B_{1+h}^c$. Thus we
 have that 
\begin{equation*}
\begin{split}
-\cJ_2 g_{\dt}(e_n,s)&=(2-\sm)\Ld\int_{|y|\ge
 1+h}\f{\mu^-(g_{\dt},e_n,y,s)}{|y|^{n+\sm}}\,dy\\
 &\le(2-\sm_0)\Ld\f{2 e^{-\bt\mu^{\sm} s-\tilde{\ap}\gm}}{(4\pi s)^{\gm n/\sm}}\int_{|y|\ge
 1+h}\f{c}{|y|^{n+\sm_0}}\,dy.
\end{split}
\end{equation*}
 Since $\sm_0\in(0,2)$, we see that
 $\ds\cJ_2 g_{\dt}(e_n,s)\ge-c_0\f{e^{-\bt\mu^{\sm} s-\tilde{\ap}\gm}}{(4\pi s)^{\gm n/\sm}}$ for a constant $c_0>0$ depending only on
 $\sm_0,\Ld$ and the dimension $n$. On the other hand, from the geometric observation
 we see that 
\begin{equation*}\begin{split}
\mu^+(g_{\dt},e_n,y,s)&=\mu(g_{\dt},e_n,y,s)+\mu(g_{\dt},e_n,-y,s)\\&\ge\mu(g_{\dt},e_n,y,s)\ge 0
\end{split}
\end{equation*}
 for any $y\in B_1$. Since the following inequality
\begin{equation}
\begin{split}
\mu(g_{\dt},e_n,y,s)&=\f{e^{-\bt \mu^{\sm} s}}{(4\pi s)^{\gm
 n/\sm}}\bigl[e^{-\tilde{\ap}\gm|e_n+y|^{\sm}}-e^{-\tilde{\ap}\gm}\bigl(1-\tilde{\ap}\gm\sm
 y_n\chi_{B_1}(y)\bigr)\bigr]\\
&\ge\f{e^{-\bt}\mu^{\sm} s}{(4\pi s)^{\gm
 n/\sm}}\bigl[e^{-\tilde{\ap}\gm|e_n+y|^{\sm}}+e^{-\tilde{\ap}\gm}(\dt\tilde{\ap}\gm\sm-1)\bigr]\\
&\ge\f{e^{-\bt}\mu^{\sm} s}{(4\pi s)^{\gm
 n/\sm}}\bigl[e^{-\tilde{\ap}\gm (1-\delta/2)^{\sm}}+e^{-\tilde{\ap}\gm}(\dt\tilde{\ap}\gm\sm-1)\bigr]\\
&\ge\f{e^{-\bt}\mu^{\sm} s}{2(4\pi s)^{\gm
 n/\sm}}\bigl[e^{-\tilde{\ap}\gm (1-\delta/2)^{\sm}}\bigr]
\end{split}
\end{equation}
 holds for any $y$ with
 $\dt/2<|y-5\dt/2|<3\dt/2$ and $y\cdot e_n>\frac12 |y|$, if we set $\dt=1/(\gm\sm)$ then we have that
 \begin{equation*}
\begin{split}
\cJ_1 g_{\dt}(e_n,s)&\ge\f{(2-\sm)\ld e^{-\bt\mu^{\sm} s}\mu^{\sm} s}{(4\pi s)^{\gm
 n/\sm}}\int_{\f{\dt}{2}<|y-\f{5\dt}{2}|<\f{3\dt}{2}}\f{e^{-\tilde{\ap}\gm|e_n+y|^{\sm}}}{|y|^{n+\sm}}\,dy\\
 &\ge\f{(2-\sm)\ld}{(4\pi )^{\gm
 n/\sm}( s)^{\gm
 n/\sm -1}}\f{e^{-\bt \mu^{\sm} s-\tilde{\ap}\gm(1-\dt/2)^{\sm}}}{(4\dt)^{n+\sm}}\biggl[\biggl(\f{3\dt}{2}\biggr)^n-\biggl(\f{\dt}{2}\biggr)^n\biggr]\\
 &=(2-\sm)\ld\,\f{e^{-\bt \mu^{\sm} s-\tilde{\ap}\gm(1-\dt/2)^{\sm}}}{(4\pi )^{\gm
 n/\sm}( s)^{\gm
 n/\sm -1}}\f{1}{4^{n+\sm}\dt^{\sm}}\f{3^n
 -1}{2^n}.
\end{split}
\end{equation*}
If we select some sufficiently large $\gm>1$ so that
 $\dt=1/(\gm\sm)$ is very small and $\ds\cJ_1
g_{\dt}(e_n,s)>c_0\,\f{e^{-\bt\mu^{\sm} s-\tilde{\ap}\gm}}{(4\pi )^{\gm
 n/\sm}( s)^{\gm
 n/\sm -1}}$, then
we can complete the proof by taking
 $\psi^{\dt}(x,t)=\min(\max(f_{\dt}(x,t)-\eta,0),At)$ for a small $\eta>0$ and a large $A>0$ as Lemma \ref{lem-barrier-1}.\qed

\begin{lemma}\label{lem-barrier-3} Let $Q=B_1\times(0,1]$, $\pa^*_p
Q=(\BR^n\times [0,1])\s Q$ and $r\in(0,1/(3\sqrt n))$. Given any
$\sm_0\in(0,2)$, there exists a function $\Psi\in\rB(\BR^n\times [0,1])$
such that

$(a)$ $\Psi$ is continuous on $\BR^n\times [0,1]$, $(b)$ $\Psi\le 0$ on
$\pa_p Q$, $(c)$ $\Psi>2$ on $\rK^+_{3r}$,

$(d)$ $\Psi\le M$ on $\BR^n\times [0,1]$ for some $M>1$, $(e)$
$\cM^-_{\fL}\Psi$ is continuous on $Q$,

$(f)$ $\cM^-_{\fL}\Psi-\Psi_t>-\psi$ on $Q$ where $\psi$
is a positive bounded function on $\BR^n\times [0,1]$ which is supported
on $\overline {\rK_{r/2}^-}\,$, for any $\sm\in(\sm_0,2)$.
\end{lemma}

\pf We consider the function $\Psi=c\,\psi^{\dt}$ for $c>0$. And we
choose some constants $c>0$ and $\dt>0$ so that $\Psi>2$ on
$Q^+_{3r}$ and $\Psi\le M$ for some $M>1$. Since
$\Psi\in\rC_x^{1,1}(Q)$, we see that $\cM^-_{\fL}\Psi$ is continuous
on $Q$. Also by Corollary \ref{cor-barrier-2} we see that
$\cM^-_{\fL}\Psi-\Psi_t\ge 0$ on $(\BR^n\times I)\s \rK_{r/2}^-$.
Hence we complete the proof. \qed

\subsection{Key Lemmas}\label{sec-decay-key}

Now we are going to show a decay estimate of the upper level set at past time depending on the future value , which will be a key step to have geometric decay rate of the  upper level set in dyadic rectangles.

\begin{lemma}\label{lem-decay-key} Let $\sm_0\in (0,2)$ and $r_0=1/(9\sqrt{n}).$  If $\,\sm\in
(\sm_0,2)$ , then there exist some constants $\vep_0>0$, $\nu\in
(0,1)$ and $M>1$ $($depending only on $\sm_0,\ld,\Ld$ and the
dimension $n$$)$ for which if $\,u\in\rB(\BR^n\times I)$ is a
viscosity supersolution to $\cM^-_{\fL} u-u_t\le\vep_0$  on $Q$ such
that $u\ge 0$ on $\BR^n\times I$ and $\inf_{\rK^+_{3r_0}}u\le 1$,
then
 $\,|\{u\le M\}\cap \rK^-_{r_0}|\ge\nu |\rK^-_{r_0}|$
\end{lemma}
\pf We consider the function $v:=\Psi-u$ where $\Psi$ is the special
function constructed in Lemma \ref{lem-barrier-3}. Then we easily
see that $v$ is upper semicontinuous on $\overline Q$ and $v$ is not
positive on $\partial^*_p Q$. Moreover, $v$ is a viscosity subsolution
to $\cM^+_{\fL} v-v_t\ge\cM^-_{\fL}\Psi-\Psi_t -(\cM^-_{\fL}
u-u_t)\ge-(\psi+\vep_0)$ on $Q$. So we want to apply Theorem
\ref{thm-abp}
 to $v$. Let $\Gm$ be the concave envelope of $v$ in
$Q$. Since $\inf_{\rK^+_{3r_0}}u\le 1$ and
$\inf_{\rK^+_{3r_0}}\Psi>2$ , we easily see that $M_0:=\sup_{
\rK^+_{3r_0}}v=v(x_0)>1$ for some $x_0\in \rK^+_{3r_0}$. We also
observe as shown in \cite{CC} that
\begin{equation}\label{eq-lem-decay-key-0}
\bigl|\n\Gm\bigl(Q\s\cC(v,\Gm,Q),t\bigr)\bigr|=0
\end{equation}
for each $t$. Let $\{K_j\}$ be the family of cubes given by Theorem
\ref{thm-abp}  with $0<r_j<\frac{r_0}{10^n}2^{-\f{1}{2-\sm}}$. Then
it follows from \eqref{eq-lem-decay-key-0} and Theorem \ref{thm-abp}
that
\begin{equation}\label{eq-decay-key-1}\begin{split}1&<\bigl(\sup_{Q}v\bigr)^{n+1}\le
C\,\int_{\cC(v,\Gm,Q) }\partial_t\Gm (y,s)\det[D^2\Gm(y,s)]^-\,dyds\\
&\qquad\le C\biggl(\sum_j\biggl(\sup_{\overline
K_j}(\psi+\vep_0)^{n+1} +r_j\sup_{Q}(\psi+\vep_0)^{n+1} \biggr)|K_j\cap\cC(v,\Gm,Q)|\biggr)\\
&\qquad\le C\vep_0+ C\biggl(\sum_j(\sup_{\overline K_j}\psi)^{n+1}
|K_j\cap\cC(v,\Gm,Q)|\biggr)+\frac14
\end{split}\end{equation} for a snall $r_0$ and some universal constant $C>0$.
 If we choose $\vep_0$ small enough, the
above inequality \eqref{eq-decay-key-1} implies that
$$\f{1}{2^{1/(n+1)}} \le C\biggl(\sum_j(\sup_{\overline
K_j}\psi)^{n+1} |K_j\cap\cC(v,\Gm,Q)|\biggr)^{1/(n+1)}.$$ We recall
from the proof of Lemma \ref{lem-barrier-3} that $\psi$ is supported
on $\overline{ \rK^-_{r_0/2}}$ and bounded on $\BR^n\times I$. Thus
the above inequality becomes $$\f{1}{2} \le
C\biggl(\,\sum_{K_j\cap\overline{\rK^-_{r_0/2}}\neq\phi}
|K_j\cap\cC(v,\Gm,Q)|\,\biggr),$$ which provides a lower bound for
the sum of the volumes of the cones $K_j$ intersecting $\overline{
\rK^-_{r_0/2}}$ as follows;
\begin{equation}\label{eq-decay-key-2}\sum_{K_j\cap
\overline{ \rK^-_{r_0/2}}\neq\phi} |K_j\cap\cC(v,\Gm,Q)|\ge \nu
|\rK^-_{r_0}|.
\end{equation}
Since
$\diam(K_j)\le\frac{r_0}{10^n}2^{-\f{1}{2-\sm}}\le\frac{r_0}{10^n}$
for any $\sm\in(\sm_0,2)$, each cone $K_j$ is contained in
$\rK^-_{r_0}$ for any $K_j$ with $K_j\cap \rK^-_{r_0/2}\neq\phi$.
Thus by Lemma \ref{lem-barrier-3} we have that
$$K_j\cap\cC(v,\Gm,Q)\subset \{v\geq 0\}=\{u\leq \Psi\}\subset \{u\leq M\}$$
for all $j$ with $K_j\cap\overline{ \rK^-_{r_0/2}}\neq\phi$. Hence
we conclude $\,|\{u\le M\}\cap \rK^-_{r_0}|\ge\nu
|\rK^-_{r_0}|$.\qed

\begin{definition} $(i)$ We say that $u$ has a tangent paraboloid of aperture $h>0$ below at
$(x_0,t_0)$ if there is a quadratic polynomial $$P(x,t)=c+b_i
(x^i-x_0^i)+h(-\frac12 |x-x_0|^2+(t_0-t))$$ such that
\begin{equation}
\begin{cases}
u(x,t)\geq P(x,t)\,\text{in $Q\cap \{t\leq t_0\}$}\\
u(x_0,t_0)=P(x_0,t_0).
\end{cases}
\end{equation}

$(ii)$ We denote by $\cG^u_h$  the set of points where $u$ has
global tangent paraboloid of aperture $h>0$ from below.
          And we set $\cB^u_h=K_1\s \cG^u_h$.
\end{definition}
From the same reason as \cite{W1}, we will have the following corollary.
\begin{cor} \label{cor-decay-key}
Under the same condition as Lemma \ref{lem-decay-key}, there exist
some constants $\vep_0>0$, $\nu\in (0,1)$ and $M>1$ $($depending
only on $\sm_0,\ld,\Ld$ and the dimension $n$$)$ for which if
$\,u\in\rB(\BR^n\times I)$ is a viscosity supersolution to
$\cM^-_{\fL} u-u_t\le\vep_0$  on $Q$ such that $u\ge 0$ on
$\BR^n\times I$ and $\inf_{\rK^+_{3r_0}}u\le 1$, or
$\rK^+_{3r_0}\cap \cG^u_1\neq \emptyset$,  then
 $\,|\cG^u_M\cap \rK^-_{r_0}|\ge\nu |\rK^-_{r_0}|$
\end{cor}

\subsection{ Nonlocal Parabolic Calder\'on-Zygmund decomposition}\label{sec-decay-cz}

The parabolic version of Calder\'on-Zygmund decomposition has been
introduced at \cite{W1}. The main difference between elliptic and
parabolic version lies in the fact that the parabolic version
requires a time interval for some information to propagate along the
time through space-time scale which is invariant under the parabolic
equation, while the elliptic version have no such time delay since
the stay state describes the behavior of solution after infinite
time. To detect the influence of the  time variable, for a given
cube $\rK$, two different associated sets will be introduced: the
time elongation and the expansion of $\rK$ along time.
\begin{definition}
For a cube $\rK=(-r,r)^n\times (0,r^{\sm}]+(x,t)$ in
space-time variable, set $l(\rK)$ to be the length of $\rK$ in the
time variable.

$(i)$ The elongation $\overline{\rK}^m$ of $\rK$ along time in $m$
steps is defined by
                    $$ \overline{\rK}^m=\bigcup_{i=0}^m \biggl((-r,r)^n\times  \Big(\frac{3^{\sigma i}-1}{3^{\sigma }-1}
                     r^{\sigma},\frac{3^{\sigma (i+1)}-1}{3^{\sigma }-1}r^{\sigma}\Big]\biggr)+(x,t)$$

$(ii)$ The expansion $\tilde{\rK}^m$ of $\rK$
 along time in $m$  steps is defined by
                     $$\tilde{\rK}^m=\bigcup_{i=0}^m\biggl((-3^ir,3^ir)^n\times  \Big(\frac{3^{\sigma i}-1}{3^{\sigma }-1}
                     r^{\sigma},\frac{3^{\sigma (i+1)}-1}{3^{\sigma }-1}r^{\sigma}\Big]\biggr)+(x,t).$$
\end{definition}

From the definition we have $\overline{\rK}^m\subset\tilde{\rK}^m$.
We divide $\rK^-_1$ step by step. First notice that there is an
increasing sequence of rational numbers $\frac{p_m}{m}$ such that
$\frac{p_m}{m}\rightarrow\sigma$. Set $N_m=2^{p_m}\times 2^{m}$ for
$0<\sigma\leq 1$,  and $2^{p_m}\times 2^{p_m+m(m-p_m)}$ for
$1<\sigma<2$.  After $m^{th}$ step, we obtain cubes $\rK_r$ of the
form $\rK_r=\rK^-_r(x,t)$. Then we split $\rK_r$ into $N_m$ cubes by
dividing the time interval into $2^{m}$ subintervals for
$0<\sigma\leq 1$ and $ 2^{p_m+m(m-p_m)}$ subsintervals for
$1<\sigma<2,$ and the rectangle in $x$-variable into equal $2^{np_m}
$-subcubes. We do the same splitting step with each one of these
$N_m$ cubes and we continue this process. The cubes obtained in this
way are called {\it dyadic cubes.} If $\rK$ and $\overline {\rK}$
are two dyadic cubes, then we say that $\overline{\rK}$ is the {\it
predecessor} of $\rK$ if $\rK$ is one of $N$ cubes obtained from
splitting $\overline{\rK}$. Now we have a parabolic version of
Calder\'on-Zygmund decomposition.

\begin{lemma}[Lemma 3.23, \cite{W1}]\label{lem-cz}
Let $\cA\subset \rK^-_1=(-1,1)^n\times (0,1]$ be a measurable set.
For $\delta\in(0,1)$, we set
$$\cA^m_{\delta}=\cup\{\overline{\rK}^m:|\rK\cap \cA|\geq \delta|\rK|,\,
\rK\,\text{dyadic cubes}\}\cap\{|x^i|\leq 1\}.$$ Then we have that
$$|\cA_{\delta}^m|\geq \frac{m}{(m+1)\delta}|\cA|.$$
\end{lemma}
We need the following version of Calderon-Zygmund decomposition in
the following form.

\begin{cor}\label{cor-cz}
Let $\fB\subset \rK^-_1=(-1,1)^n\times (0,1]$ be a measurable set.
For $\delta\in(0,1)$ and any dyadic cube $\rK=\rK^-_r(x,t)$, we set
$\overline{\rK}_*^m=(x+(-r,r)^n)\times
(t,t+\f{3^{\sm(m+1)}-1}{3^{\sm}-1}r^\sm]$. Let
$$\fB^m_{\delta}=\cup\{\overline{\rK}_*^m:|\rK\cap \fB|\geq \delta|\rK|,\,
\rK\,\text{dyadic cubes}\}\cap\{|x^i|\leq 1\}.$$ Then we have that
$$|\fB_{\delta}^m|\geq \frac{m}{(m+1)\delta}|\fB|.$$
\end{cor} We note that $\overline{\rK}^m_*=\overline{\rK}^m$ for any
dyadic cube $\rK=\rK^-_r(x,t)$.

\subsection{Decay estimate through  iterations}\label{sec-decay-it}
The following lemma is a consequence of Lemma \ref{lem-decay-key}
and Lemma \ref{lem-cz}.

\begin{lemma}\label{lem-decay} Given $\sm_0\in (0,2)$, let $\,\sm\in
(\sm_0,2)$. Let $\vep_0, r_0$ be the constants in Lemma
\ref{lem-decay-key}. If $u\in\rB(\BR^n\times I)$ is a viscosity
supersolution to $\cM^-_{\fL}u-u_t\le\vep_0$  on $Q$ such that $u\ge
0$ on $\BR^n\times I$ and $\inf_{\rK^+_{3r_0}}u\le 1$, then there
are universal constants $C>0$ and $\vep_*>0$ such that
$$\bigl|\cB^u_s\cap \rK^-_{r_0}\bigr|\le C\,s^{-\vep_*}|\rK^-_{r_0}|,\,\forall\,s>0.$$
\end{lemma}

\pf (i) First, we shall prove that there is $\beta\ge 1$ such that
 \begin{equation}\label{eq-decay-1}
 \bigl|\cB^u_{M^{\beta k}}\cap
\rK^-_{r_0}\bigr|\le(1-\nu/2)^k|\rK^-_{r_0}|,\,\forall\,k\in\BN,\end{equation}
where $\nu>0$ is the constant as in Lemma \ref{lem-decay-key}. We
will prove it by induction. The case $k=1$ has been proved at Lemma
\ref{lem-decay-key} for $\beta\geq 1$. Assume that the result
\eqref{eq-decay-1} holds for $k$ ($k\ge 1$). Set  $\fB=\cB^u_{
M^{\beta (k+1)}}\cap \rK^-_{r_0}$. We choose a large $m$ such that
$1>1-\nu/2>\frac{(m+1)(1-\nu)}{m}>0$ and consider $\fB^m_{\delta}$
at Corollary \ref{cor-cz} with $\delta=1-\nu$. We claim that
\begin{equation}\label{eq-decay-2}
\fB^m_{\delta}\subset \cB^u_{M^{\beta k}}\cap \rK^-_{r_0}
\end{equation}
 for a uniform constant $\beta\ge 1$ independent of $k$.
 If the claim is true and if $|\fB|>\left(1-\nu/2\right)^{k+1}|\rK^-_{r_0}|$ ,
 then we have that
 \begin{equation*}
 |\cB^u_{M^{\beta k}}\cap \rK^-_{r_0}|\geq |\fB^m_{\delta}|\geq  \frac{m}{(m+1)\delta}|\fB|
>\frac{\left(1-\nu/2\right)^{k+1}}{1-\nu/2}|\rK^-_{r_0}|=(1-\nu/2)^k|\rK^-_{r_0}|,
 \end{equation*}
which is a contradiction against the assumption. Therefore we have
that $|\cB^u_{ M^{\beta (k+1)}}\cap \rK^-_{r_0}|=|\fB|\leq
\left(1-\nu\right)^{k+1}|\rK^-_{r_0}|.$

(ii)  Now we are going to show the claim. If the claim is not rue,
there is $(x_0,t_0)\in  \fB^m_{\delta}\backslash\Big(\cB^u_{M^{\beta
k}}\cap \rK^-_{r_0}\Big) ,$ which implies that there is a dyadic
cube $\rK=\rK^-_r$ such that $(x_0,t_0)\in{\overline \rK}^m_*$ and
\begin{equation}\label{eq-decay-3}
|\fB\cap \rK|\geq \delta |\rK|
\end{equation}
and $(x_0,t_0)\in \cG^u_{M^{\beta k}}\cap
\Big(\overline{\rK}^j_*\s\overline{\rK}_*^{j-1}\Big)=\cG^u_{M^{\beta
k}}\cap \Big(\overline{\rK}^j_*\s\tilde{\rK}^{j-1}\Big)$ for some
$1\leq j\leq m$. By the definition of an expansion  $\tilde{\rK}^j$
of $\rK$ along the time, we have $(x_0,t_0)\in
\overline{\rK}^j_*\subset\tilde{\rK}^j$.

(iii) Now we are going to apply Lemma \ref{lem-decay-key} backward
from $j$ to $1$ after scaling to show that
\begin{equation}\label{eq-decay-4}
|\cG^u_{M^iM^{\beta k}}\cap \tilde{\rK}^{j-i}|\geq \nu
3^{(j-i)(n+\sm)}|\rK|
\end{equation}
for $i=1,\cdots,j$. If \eqref{eq-decay-4} is true, then since
$\tilde{\rK}^0=\rK$ we have
 \begin{equation*}
|\rK\s\fB |=|\cG^u_{M^{\beta (k+1)}}\cap \rK|>|\cG^u_{M^jM^{\beta
k}}\cap \tilde{\rK}^0|\geq \nu |\rK|
\end{equation*}
for $\beta=m+1$ independent of $k$. Now  we have a contradiction
against \eqref{eq-decay-3} since  \begin{equation*} |\rK\cap
\fB|<(1- \nu) |\rK|=\delta |\rK|
\end{equation*}
and then the claim \eqref{eq-decay-1}.

(iv) Now we are going to show \eqref{eq-decay-4}. Assume that
\begin{equation}\begin{split}\label{eq-decay-5}
&\cG^u_{M^{i-1}M^{\beta k}}\cap \left(\tilde{\rK}^{j-i+1}\backslash
\tilde{\rK}^{j-i}\right)\\&\qquad=\cG^u_{M^{i-1}M^{\beta k}}\cap
\left[\bigl(\tilde{\rK}^{j-i+1}\s\tilde{\rK}^{j-i-1}\bigr)\s\bigl(\tilde{\rK}^{j-i}\s\tilde{\rK}^{j-i-1}\bigr)\right]\neq\emptyset,
\end{split}\end{equation}
which is true for $i=1$. Then there is some $
(y_0,s_0)\in\tilde{\rK}^{j-i+1}\s\tilde{\rK}^{j-i-1}$ such that
 $\tilde{\rK}^{j-i}\s\tilde{\rK}^{j-i-1}=\rK^-_{3^{j-i}r}(y_0,s_0)$.
Now we consider the transformation
$$x=y_0+(3^{j-i}r)z,\,\,t=s_0+(3^{j-i}r)^{\sm}\tau$$ and the function
$v(z,\tau)=u(x,t)/M^{\beta k+i-1}$ for $(z,\tau)\in \rK_1^-$. Notice
that $\tilde{\rK}^{j-i+1}\backslash \tilde{\rK}^{j-i-1}$ and $
\tilde{\rK}^{j-i}\backslash \tilde{\rK}^{j-i-1}$ will be transformed
to $\rK^+_{3r}$ and $\rK^-_r$ under the transformation.
To complete the proof, it now remains to show that $v$ satisfies the
hypothesis of Corollary \ref{cor-decay-key}. We now take any
$\vp\in\rC^2_{Q}(v;z,\tau)^-$. If we set
$\psi(z,\tau)=\vp(x,t)/M^{\beta k+i-1}$, then we observe that
$$\vp\in\rC^2_{Q}(v;z,\tau)^-\,\,\Leftrightarrow\,\,\psi\in
\rC^2_{Q_{3^{j-i}r}(y_0,s_0)}(u;x,t)^-.$$ Now we have that
\begin{equation*}\begin{split}&\cM^-_{\fL}v(z,\tau)-\vp_\tau(z,\tau)\le\cL
v(z,\tau)-\vp_\tau(z,\tau)\\
&\qquad=\frac{1}{ (3^{j-i}r )^{\sigma}M^{\beta k+i-1}}\left(\int_{\BR^n}\mu (u,x,t,y)K(y)\,dy-\psi_t(x,t)\right)\\
&\qquad=\frac{1}{ (3^{j-i}r )^{\sigma}M^{\beta k+i-1}}\left(\cL
u(x,t)-\psi_t(x,t)\right)
\end{split}\end{equation*} for any $\cL\in\fL$. Since
$Q_{3^{j-i}r}(y_0,s_0)\subset Q$, this implies that
$$\cM^-_{\fL}v(z,\tau)-\vp_\tau(z,\tau)\le\frac{1}{ (3^{j-i}r
)^{\sigma}M^{\beta
k+i-1}}\bigl(\cM^-_{\fL}u(x,t)-\psi_t(x,t)\bigr)\le\vep_0.$$ Also it
is obvious that $v\ge 0$ on $\BR^n$, and thus we see from
\eqref{eq-decay-5} that  $\rK^+_{3r}\cap \cG^v_1\neq \emptyset$. By
Corollary \ref{cor-decay-key}, we obtain that
$|\cG^v_M\cap\rK_r^-|\ge\nu|\rK_r^-|$. We note that
$$|\cG^v_M\cap\rK_r^-|\ge\nu|\rK_r^-|\,\,\Leftrightarrow\,\,\bigl|\cG^u_{M^i
M^{\bt
k}}\cap(\tilde{\rK}^{j-i}\s\tilde{\rK}^{j-i-1})\bigr|\ge\nu\bigl|\tilde{\rK}^{j-i}\s\tilde{\rK}^{j-i-1}\bigr|.$$
Since
$\tilde{\rK}^{j-i}\s\tilde{\rK}^{j-i-1}=\rK^-_{3^{j-i}r}(y_0,s_0)$,
this implies that $$\bigl|\cG^u_{M^i M^{\bt
k}}\cap\tilde{\rK}^{j-i}\bigr|\ge\nu\bigl|\rK^-_{3^{j-i}r}(y_0,s_0)\bigr|=\nu
3^{(j-i)(n+\sm)}|\rK|.$$

(v)  Finally the result follows immediately from \eqref{eq-decay-1}
by taking $C=(1-\nu)^{-1}$ and $\vep_*>0$ so that $1-\nu=(M^{\beta
})^{-\vep_*}$. Hence we complete the proof. \qed

By a standard covering argument we obtain the following theorem.

\begin{thm}\label{thm-decay} For any $\sm_0\in (0,2)$, let
$\sm\in(\sm_0,2)$ be given. If $u\in\rB(\BR^n\times I)$ is a
viscosity supersolution to $\cM^-_{\fL} u-u_t\le \vep_0$ with
$\,\sm\in (\sm_0,2)$ on $\rQ_2$ such that $u\ge 0$ on $\BR^n\times
I$ and $u(0,0)\le 1$ where $\vep_0$ is the constant given in Lemma
\ref{lem-decay-key}, then there are universal constants $C>0$ and
$\vep_*>0$ such that
$$\bigl|\{u>s\}\cap\rQ_1\bigr|\le C\,s^{-\vep_*} |\rQ_1|,\,\forall\,s>0.$$\end{thm}

\begin{thm}\label{thm-harnack-weak} Given $\sm_0\in (0,2)$, let
$\sm\in(\sm_0,2)$, and let $(x_0,t_0)\in\BR^n\times I$ and $r\in
(0,2]$. If $\,u\in\rB(\BR^n\times I)$ is a viscosity supersolution
to $\cM^-_{\fL} u-u_t\le c_0$ on $\rQ_{2r}(x_0,t_0)$ such that $u\ge
0$ on $\BR^n\times I$, then there are universal constants $\vep_*>0$
and $C>0$ such that
$$\bigl|\{u>s\}\cap\rQ_r(x_0,t_0)\bigr|\le C\,r^{n+\sm}\bigl(u(x_0,t_0)+c_0\,r^{\sm}\bigr)^{\vep_*}
s^{-\vep_*},\,\forall\,s>0.$$\end{thm}

\pf Let $(x_0,t_0)\in\BR^n\times I$ and set
$v(z,\tau)=u(rz+x_0,r^{\sm}\tau+t_0)/q$ for $(z,\tau)\in\rQ_{2}$
where $q=u(x_0,t_0)+c_0 r^{\sm}/\vep_0.$ Take any $\vp\in
\rC^2_{\rQ_{2}}(v;z,\tau)^-$. If we set
$\psi=q\,\vp(\f{\,\cdot\,-x_0}{r},\f{\,\cdot\,-t_0}{r^{\sm}})$, then
we see that
$\psi\in\rC^2_{\rQ_{2r}(x_0,t_0)}(u;rz+x_0,r^{\sm}\tau+t_0)$. Thus
by the change of variables $x=rz+x_0$ and $r^{\sm}\tau+t_0$, we have
that
\begin{equation*}\begin{split}&\cM^-_{\fL}v(z,\tau)-\vp_{\tau}(z,\tau)\le\cL
v(z,\tau)-\vp_{\tau}(z,\tau)\\
&\quad=\f{r^{\sm}}{q}\biggl(\int_{\BR^n}\mu(u,rz+x_0,y,r^{\sm}\tau+t_0)\,K(y)\,dy-\psi_t(rz+x_0,r^{\sm}\tau+t_0)\biggr)\\
&\quad:=\f{r^{\sm}}{q}\bigl(\cL
u(x,t)-\psi_t(x,t)\bigr)\end{split}\end{equation*} for any
$\cL\in\fL_0$. Taking the infimum of the right-hand side in the
above inequality, we get that
\begin{equation*}\cM^-_{\fL}v(z,\tau)-\vp_{\tau}(z,\tau)\le\f{r^{\sm}}{q}\biggl(\cM^-_{\fL}
u(x,t)-\psi_t(x,t)\biggr)\le \vep_0.\end{equation*} Thus we have
that $\cM^-_{\fL}v-\vp_{\tau}\le\vep_0$ on $\rQ_{2}$. Applying
Theorem \ref{thm-decay} to the function $v$, we complete the proof.
\qed

\section{Regularity Theory}
\label{sec-regularity}

\subsection{ H\"older estimates}\label{sec-regularity-harnack} In this subsection, we
obtain H\"older regularity result. The following technical lemma is
very useful in proving it. As in \cite{CS,KL1}, its proof can be
derived from Theorem \ref{thm-harnack-weak}.

In this subsection, we are going to take a notation
$\bQ_r\fd\rQ_r\cup Q_r=B_r\times(-r^{\sm},r^{\sm}]$ and
$\bQ_r(x_0,t_0)=\bQ_r+(x_0,t_0)$ for any $r>0$ and
$(x_0,t_0)\in\BR^n\times I$.

\begin{lemma}\label{lem-holder} For $\sm_0\in (0,2)$, let $\sm\in
(\sm_0,2)$ be given. If $u$ is a bounded function with $|u|\le 1/2$
on $\BR^n\times I$ such that
$$\cM_{\fL}^- u-u_t\le\vep_0\,\,\,\text{ and }\,\,\,\cM_{\fL}^+ u-u_t\ge-\vep_0\,\,\text{ on $\bQ_2$ }$$
in the viscosity sense
where $\vep_0>0$ is some sufficiently small constant, then there is
some universal constant $\ap>0$ $($depending only on $\ld,\Ld,n$ and
$\sm_0$$)$ such that $u\in\rC^{\ap}$ at the origin. More precisely,
$$|u(x,t)-u(0,0)|\le C\,(|x|^{\sm}+|t|)^{\ap/\sm}$$ for some universal
constant $C>0$ depending only on $\ap$.\end{lemma}

Lemma \ref{lem-holder} and a simple rescaling argument give the
following theorem 4.19  in \cite {W1} and theorems  as in \cite{CS,
KL1, KL2} .
\begin{thm}\label{thm-holder} For any $\sm_0\in (0,2)$, let $\sm\in
(\sm_0,2)$ be given. If $u$ is a bounded function on $\BR^n\times I$
such that
$$\cM_{\fL}^- u-u_t\le C_0\,\,\,\text{ and }\,\,\,\cM_{\fL}^+ u-u_t\ge -C_0\,\,\text{
on $\bQ_2$}$$ in the viscosity sense, then there is some constant
$\ap>0$ $($depending only on $\ld,\Ld,n$ and $\sm_0$$)$ such that
$$\|u\|_{\rC^{\ap}(\bQ_{1/2})}\le
C\bigl(\,\|u\|_{L^{\iy}(\BR^n\times I)}+C_0\bigr)$$ where $C>0$ is
some universal constant depending only on $\ap$.\end{thm}

Corollary \ref{cor-harnack-weak} can be shown by the same way as
Theorem \ref{thm-harnack-weak}.

\begin{cor}\label{cor-harnack-weak} Given $\sm_0\in (0,2)$, let
$\sm\in(\sm_0,2)$, and let $(x_0,t_0)\in\BR^n\times I$ and $r\in
(0,2]$. If $\,u\in\rB(\BR^n\times I)$ is a viscosity supersolution
to $\cM^-_{\fL} u-u_t\le c_0$ on $\bQ_{2r}(x_0,t_0)$ such that $u\ge
0$ on $\BR^n\times I$, then there are universal constants $\vep_*>0$
and $C>0$ such that
$$\bigl|\{u>s\}\cap\bQ_r(x_0,t_0)\bigr|\le C\,r^{n+\sm}\bigl(u(x_0,t_0)+c_0\,r^{\sm}\bigr)^{\vep_*}
s^{-\vep_*},\,\forall\,s>0.$$\end{cor}

{\it $[$Proof of Lemma \ref{lem-holder}$]$ } We take any
$\ap\in(0,\sm_0)$ and choose some $N\ge 1$ so large that $
2^{1-\sm_0 N}2^{-k(\sm_0-\ap)N}\le 1/2$ and
\begin{equation}\label{eq-5.2.1}(2-\sm_0)\Ld\om_n\biggl(\f{2^{\sm_0+1}}{\sm_0-\ap}\,2^{-(N-1)(\sm_0-\ap)}
+\f{|\tau_1|\vee\tau_2}{\sm_0}\,2^{-(N-1)\sm_0}\biggr)\le\vep_0/2.
\end{equation} We may regard $u$ as a function on $\BR^n\times (-\infty,\tau_2)$
by setting $u=u\chi_{\BR^n\times I}$. Then it is enough to show that
there is a nondecreasing sequence $\{n_k\}_{k\in\BN\cup\{0\}}$ and a
nonincreasing sequence $\{N_k\}_{k\in\BN\cup\{0\}}$ such that
$N_k-n_k=2^{-\ap kN}$ and $n_k\le u\le N_k$ in $\rQ_{2^{-kN}}$. This
implies that the theorem holds with $C=2^{\ap N}$; for, if
$2^{-(k+1)N}\le(|x|^{\sm}+|t|)^{1/\sm}\le 2^{-kN}$ for
$k\in\BN\cup\{0\}$, then we have that
$$|u(x,t)-u(0,0)|\le\f{N_k-n_k}{2^{-\ap(k+1)N}}\cdot 2^{-\ap(k+1)N}\le
2^{\ap N}(|x|^{\sm}+|t|)^{1/\sm}.$$ We now construct $n_k$ and $N_k$
by induction process. For $k= 0$, we can take
$n_k=\inf_{\BR^n\times\BR}\,u$ and $N_0=n_0+1$ because
$\osc_{\BR^n\times\BR}\,u\le 1$. We assume that we obtained the
sequences up to $n_k$ and $N_k$ for $k\ge 1$. Then we shall show
that we can continue the sequences by finding $n_{k+1}$ and
$N_{k+1}$.

Fix any $(x,t)\in\bQ_{1/\e}$ where $\e=2^{-(k+1)N}$. Take any
$\vp\in\rC^2_{\bQ_{1/\e}}(v;x,t)^-$ where
$$v=\f{u(\e\,\cdot\,,\e^{\sm}\cdot\,)-n_k}{(N_k-n_k)/2}.$$ If we set
$\psi=n_k+\f{N_k-n_k}{2}\vp(\f{\cdot}{\e},\f{\cdot}{\e^{\sm}})$,
then we see that $\psi\in\rC^2_{\bQ_{1}}(u;\e x,\e^{\sm}t)^-$.

In $\bQ_{2^{-(k+1)N}}$, we have two possible cases; either (a)
$u>(N_k+n_k)/2$ in at least half of the points (in measure) or (b)
$u\le(N_k+n_k)/2$ in at least half of the points. First, we deal
with the case (a)
$$|\{u>(N_k+n_k)/2\}\cap\bQ_{2^{-(k+1)N}}|\ge
|\bQ_{2^{-(k+1)N}}|/2.$$ Then we see that $v\ge 0$ on $\bQ_{2^N}$
and $|\{v>1\}\cap\bQ_1|\ge|\bQ_1|/2.$ We observe that the mapping
$\cK_0\mapsto\cK_0$ given by $K\to K_{\e}$ is an isometry. Thus by
the change of variables we have that
\begin{equation*}\begin{split}&\cM^-_{\fL}v(x,t)-v_{t}(x,t)\le\cL
v(x,t)-v_{t}(x,t)\\
&=\f{2\e^{\sm}}{N_k-n_k}\biggl(\int_{\BR^n}\mu(u,\e
x,\e^{\sm}t)K_{\e}(y,t)\,dy-u_{t}(\e x,\e^{\sm}t)\biggr)
\end{split}\end{equation*} for any $K\in\cK_0$. Taking the infimum
on $\fL$ of the right-hand side in the above inequality and using
the assumption that $\cM^-_{\fL}u-u_t\le\vep_0$ on $\bQ_2$ in the
viscosity sense, we obtain that
\begin{equation*}\begin{split}\cM^-_{\fL}v(x,t)-v_{t}(x,t)&\le\f{2\e^{\sm}}{N_k-n_k}
\bigl(\cM^-_{\fL}u(\e x,\e^{\sm}t)-u_{t}(\e
x,\e^{\sm}t)\bigr)\\
&\le 2^{1-\sm_0 N}2^{-k(\sm_0-\ap)N}\vep_0\le\vep_0/2.
\end{split}\end{equation*}
This implies that
\begin{equation}\label{eq-5.2.4}\cM^-_{\fL}v-v_{t}\le\vep_0/2\,\,\text{
on $\bQ_{2^{(k+1)N}}$. }\end{equation} By the induction hypothesis,
if $2^{jN}\le(|x|^{\sm}+|t|)^{1/\sm}\le 2^{(j+1)N}$, then
\begin{equation}\begin{split}\label{eq-5.2.5}
v(x,t)&\ge\f{n_{k-j}-N_{k-j}+N_k-n_k}{(N_k-n_k)/2}\ge-2(|x|^{\sm}+|t|)^{\ap/\sm},\\
v(x,t)&\le\f{N_{k-j}-n_{k-j}+n_{k-j}-n_k}{(N_k-n_k)/2}\le
2(|x|^{\sm}+|t|)^{\ap/\sm}\end{split}\end{equation} for any
$j\in\BN\cup\{0\}$. This implies that $|v(x,t)|\le
2(|x|^{\sm}+|t|)^{\ap/\sm}$ outside $\bQ_1$.

Set $w(x,t)=\max\{v(x,t),0\}$.We are going to show that
$\cM^-_{\fL}w-w_t\le 4\vep_0$ on $\bQ_{2^{N-1}}$. Notice that if
$(x,t)\in\bQ_{2^{N-1}}$ and $y\in B_{2^{N-1}}$, then
$\mu(v^-,x,y,t)=0$ because $v\ge 0$ on $\bQ_{2^N}$; if
$(x,t)\in\bQ_{2^{N-1}}$ and $y\in B^c_{2^{N-1}}$, then
$\mu(v^-,x,y,t)=v^-(x+y,t)$ because $v\ge 0$ on $\bQ_{2^N}$. Since
$w=v+v^-$, then we see that
$\cM^-_{\fL}w-w_t\le\cM^-_{\fL}v-v_t+\cM^+_{\fL}v^--v^-_t$. Thus it
follows from \eqref{eq-5.2.1} and \eqref{eq-5.2.5} that if
$(x,t)\in\bQ_{2^{N-1}}$ is given, then
\begin{equation*}\begin{split}
\cL
v^-(x,t)&\le\int_{\BR^n}\mu(v^-,x,y,t)K(y,t)\,dy\\&\le(2-\sm)\Ld\int_{|y|\ge
2^{N-1}}\f{2(|x+y|^{\sm}+|t|)^{\ap/\sm}}{|y|^{n+\sm}}\,dy\\
&\le (2-\sm_0)\Ld\biggl(\int_{|y|\ge
2^{N-1}}\f{2^{\ap(1+1/\sm)}}{|y|^{n+\sm_0-\ap}}dy+\int_{|y|\ge
2^{N-1}}\f{(|\tau_1|\vee\tau_2)^{\ap/\sm}}{|y|^{n+\sm_0}}\biggr)\\
&\le(2-\sm_0)\Ld\om_n\biggl(\f{2^{\sm_0+1}}{\sm_0-\ap}\,2^{-(N-1)(\sm_0-\ap)}
+\f{|\tau_1|\vee\tau_2}{\sm_0}\,2^{-(N-1)\sm_0}\biggr)\\
&\le\vep_0/2\end{split}\end{equation*} for any $\cL\in\fL$, whenever
$\ap\in(0,\sm_0)$ and $\sm\in(\sm_0,2)$. Since $v^-_t(x,t)=0$ for
any $(x,t)\in\bQ_{2^{N-1}}$, we have that $\cM^+_{\fL}
v^--v^-_t\le\vep_0/2$ on $\bQ_{2^{N-1}}.$ Thus by \eqref{eq-5.2.4}
we conclude that $\cM^-_{\fL}w-w_t\le\vep_0$ on $\bQ_{2^{N-1}}$
where $\ap\in(0,\sm_0)$.

We now take any point $(x,t)\in\bQ_1$. Since
$\bQ_1\subset\bQ_2(x,t)\subset\bQ_4(x,t)\subset\bQ_{2^{N-1}}$, we
can apply Corollary \ref{cor-harnack-weak} on $\bQ_2(x,t)$ to obtain
that
$$2^{n+\sm}C\bigl(w(x,t)+2^{\sm}\vep_0\bigr)^{\vep_*}\ge|\{w>1\}\cap\bQ_2(x,t)|
\ge|\{v>1\}\cap\bQ_1|\ge\f{1}{2}\,|\bQ_1|.$$ Thus we have that
$$w(x,t)\ge\biggl(\f{|\bQ_1|}{2^{n+\sm+1}C}\biggr)^{1/\vep_*}-4\vep_0$$
for any $(x,t)\in\bQ_1$. If we select $\vep_0$ sufficiently small,
then there is some $\kp>0$ such that $w\ge\kp$ on $\bQ_1$. If we
take $N_{k+1}=N_k$ and $n_{k+1}=n_k+\kp(N_k-n_k)/2$, then we have
that $n_{k+1}\le u\le N_{k+1}$ on $\bQ_{2^{-(k+1)N}}$. Furthermore,
we have that $N_{k+1}-n_{k+1}=(1-\kp/2)2^{-\ap kN}$. Now we may
choose some small $\ap>0$ and $\kp>0$ so that $1-\kp/2=2^{-\ap N}$.
Hence we obtain that $N_{k+1}-n_{k+1}=2^{-\ap(k+1)N}$.

On the other hand, if we treat of the second case (b)
$$|\{u\le(N_k+n_k)/2\}\cap\bQ_{2^{-(k+1)N}}|\ge
|\bQ_{2^{-(k+1)N}}|/2,$$ then we consider the function $v(x,t)=\ds
\f{N_k-u(2^{-(k+1)N}x,2^{-\ap(k+1)N}t)}{(N_k-n_k)/2}$ and repeat the
same way as in the above by using $\cM^+_{\fL}u-u_t\ge-\vep_0$. \qed

\subsection{ $\rC^{1,\ap}$-estimates}\label{sec-regularity-c1a}

In this subsection, we establish an interior
$\rC^{1,\ap}$-regularity result for viscosity solutions as \cite{CS}.
We now consider the class $\fL^1$ consisting of the operators
$\cL\in\fL$ associated with kernel $K$ for which
\eqref{eq-kernel-bound} holds and there exists some $\e_1>0$ such
that
\begin{equation}\label{eq-extra}\sup_{t\in I}\sup_{h\in B_{\e_1}}\int_{\BR^n\s
B_{2\e_1}}\f{|K(y,t)-K(y-h,t)|}{|h|}\,dy\le C.\end{equation}

\begin{thm}\label{thm-regularity-c1a} For $\sm_0\in (0,2)$, let $\sm\in
(\sm_0,2)$ be given. Then there is some $\e_1>0$ $($depending on
$\ld,\Ld,\sm_0$ and the dimension $n$$)$ so that if $\,\cI$ is a
nonlocal elliptic operator with respect to $\fL^1$ in the sense of
Definition \ref{def-class} and $u\in\rB(\BR^n\times I)$ is a
viscosity solution to $u_t=\cI u$ on $\bQ_2$, then there is a
universal constant $\ap>0$ $($depending only on $\ld,\Ld,\sm_0$ and
the dimension $n$$)$ such that
$$\|u\|_{\rC^{1,\ap}(\bQ_{1/2})}\le C\bigl(\,\|u\|_{L^{\iy}(\BR^n\times I)}+|\cI 0|\,\bigr)$$
for some constant $C>0$ depending on $\ld,\Ld,\sm_0,n$ and the
constant given in \eqref{eq-extra} $($where we denote by $\cI 0$ the
value we obtain when we apply $\cI$  with $0<\sigma< 2$ to the
constant function that is equal to zero$)$.\end{thm}

\pf Since $u_t=\cI u$ on $\bQ_2$, by Definition \ref{def-class} we have that
$\cM^+_{\fL}u\ge\cI u-\cI 0=u_t-\cI 0\ge u_t-|\cI 0|$, and so
$\cM^+_{\fL}u-u_t\ge-|\cI 0|$ on $\bQ_2$. Similarly we have that
$\cM^-_{\fL}u-u_t\le|\cI 0|$ on $\bQ_2$. Thus it follows from
Theorem \ref{thm-holder} that $u\in\rC^{\ap}(\bQ_{1-\dt})$ for any
$\dt\in(0,1)$ and $\|u\|_{\rC^{\ap}(\bQ_{1-\dt})}\le
C(\|u\|_{L^{\iy}(\BR^n\times I)}+|\cI 0|)$. Now we will try to
improve the obtained regularity iteratively by applying Theorem
\ref{thm-holder} again until we reach Lipschitz regularity in a
finite number of steps.

Assume that we have shown that $u\in\rC^{\bt}(\bQ_r)$ for some
$\bt\in(0,1]$ and $r\in(0,1)$. Then we apply Theorem
\ref{thm-holder} to the difference quotient $u^h=(\tau_h
u-u)/|h|^{\bt}$ where $\tau_h$ is a translation operator in space
variable given by $\tau_h u(x,t)=u(x+h,t)$ for $h\in\BR^n$ and $t\in
I$. Since we see from Theorem 2.0.4 that $\cM^+_{\fL}u^h-u^h_t\ge 0$
and $\cM^-_{\fL}u^h-u^h_t\le 0$ on $\bQ_r$ (in the viscosity sense)
for any $h$ with $|h|\in(0,1-r)$, it follows from Theorem
\ref{thm-holder} that $u^h\in\rC^{\bt}(\bQ_r)$ and the family
$\{u^h\}_{|h|\in(0,1-r)}$ is uniformly bounded on $\bQ_r$ with bound
$C\|u\|_{L^{\iy}(\BR^n\times I)}$. However the functions $u^h$ is
not uniformly bounded outside the ball $\bQ_r$, and thus we can not
directly apply Theorem \ref{thm-holder}. But we have a nice tool
\eqref{eq-extra} to overcome this obstacle. For this purpose, we
employ a smooth cutoff function $\phi$ supported in $\bQ_r$ such
that $\phi\equiv 1$ in $\bQ_{r-\dt/4}$ where $\dt>0$ is some small
positive number to be determined later. We write $u^h=v^h+w^h$ where
$v^h=\phi\,u^h$ and $w^h=(1-\phi)u^h$.

Take any $(x,t)\in\bQ_{r-\dt/2}$ and $|h|<\dt/16$. Then we see
$(1-\phi(x,t))u(x,t)=(1-\phi(x,t))\tau_h u(x,t)=0$ and
$u^h(x,t)=v^h(x,t)$. We shall now prove
$v^h\in\rC^{\ap+\bt}(\bQ_{r-\dt})$ for some $\ap>0$ with
$\ap+\bt>1$. From Definition \ref{def-class}, we can derive the
following inequalities;
\begin{equation}\begin{split}\label{eq-ineq}
\cM^+_{\fL}v^h-v_t&\ge-\cM^+_{\fL^1}w^h+w_t-\cI 0=-\cM^+_{\fL^1}w^h-\cI 0,\\
\cM^-_{\fL}v^h-v_t&\le-\cM^-_{\fL^1}w^h+w_t-\cI
0=-\cM^-_{\fL^1}w^h-\cI 0
\end{split}\end{equation} on $\bQ_{r-\dt/2}$, because $w_t\equiv 0$
on $\bQ_{r-\dt/2}$. In order to apply Theorem \ref{thm-holder}, we
must show that $|\cM^+_{\fL^1}w^h|$ and $|\cM^-_{\fL^1}w^h|$ are
bounded on $\bQ_{r-\dt/2}$ by $C\|u\|_{L^{\iy}(\BR^n\times I)}$ for
some universal constant $C>0$. To show this, we have only to prove
that it is true for any operator $\cL\in\fL^1$. Take any
$\cL\in\fL^1$. Since $w_t\equiv 0$ on $\bQ_{r-\dt/2}$,the conclusion comes from the argument on the elliptic case as \cite{CS}\qed


\subsection{ Harnack inequality}
\label{sec-regularity-harnack}

Now we are going to show Harnack inequality .

\begin{thm}\label{thm-harnack-bd} Given $\sm_0\in (0,2)$, let
$\sm\in(\sm_0,2)$. If $\,u\in\rB(\BR^n\times I)$ is a positive
function such that
$$\cM_{\fL}^- u-u_t\le C_0\,\,\,\text{ and }\,\,\,\cM_{\fL}^+ u-u_t\ge -C_0\,\,\text{
on $\rQ_{2}$}$$
 in the viscosity sense, then there is some constant
$C>0$ depending only on $\ld,\Ld,n$ , $\sm_0$, $\|u\|_{L^{\infty}( \BR^n\times I)}$ such that
$$\sup_{\rQ^-_{1/2}} u\le C\,\bigl(\,\inf_{\rQ^+_{1/2}}u+C_0 \bigr).$$\end{thm}

\pf Let $(\hat x,\hat t)\in\rQ^+_{1/2}$ be a point so that
$\inf_{\rQ^+_{1/2}}u=u(\hat x,\hat t)$. Then it suffices to show that
$$\sup_{\rQ^-_{1/2}} u\le C \,\bigl(u(\hat x,\hat t)+C_0 \bigr).$$ Without loss of
generality, we may assume that $u(\hat x,\hat t)\le 1$ and $C_0=1$
by dividing $u$ by $u(\hat x,\hat t)+C_0 $. Let $\vep_*>0$ be the
number given in Theorem \ref{thm-harnack-weak} and let
$\bt=(n+\sigma)/\vep_*$. We now set $s_0=\inf\{s>0:u(x,t)\le s\,
\dd((x,t),\pa\rQ_1)^{-\bt},\,\forall\,(x,t)\in\rQ_1\}.$ Then we see
that $s_0>0$ because $u$ is positive on $\BR^n\times I$. Also there
is some $(\check x,\check t)\in\rQ_1$ such that $u(\check x,\check
t)=s_0\,\dd((\check x,\check t),\pa\rQ_1))^{-\bt}=s_0\dd_0^{-\bt}$
where $\dd_0=\dd((\check x,\check t),\pa\rQ_1)\le 2^{1/\sm}<2$ for
$\sm\in(1,2)$. We note that
\begin{equation}\label{subset}
\rB^{\dd}_r(x_0,t_0)\subset
Q_r(x_0,t_0)\subset\rB^{\dd}_{2r}(x_0,t_0)\end{equation} for any
$r>0$ and $(x_0,t_0)\in\BR^n\times I$.

To finish the proof, we have only to show that $s_0$ can not be too
large because $u(x,t)\le C_1 \dd((x,t),\pa\rQ_1)^{-\bt}\le C$ for
any $x\in\rQ^-_{1/2}\subset\rQ_1$ if $C_1>0$ is some constant with
$s_0\le C_1$. Assume that $s_0$ is very large. Then by Theorem
\ref{thm-decay} we have that
$$\bigl|\{u\ge u(\check x,\check t)/2\}\cap
\rQ_1\}\bigr|\le\biggl|\f{2}{u(\check x,\check t)}\biggr|^{\vep_*}
|\rQ_1|\le C s_0^{-\vep_*}\dd_0^{n+\sm}.$$ Since $\rB^{\dd}_r(\check
x,\check t)\subset\rQ_1$ and $|\rB^{\dd}_{r}|=C\dd_0^{n+\sigma}$ for
$r=\dd_0/2\le 2^{-(1-1/\sm)}<1$ for $\sm\in(1,2)$, we easily obtain
that
\begin{equation}\label{eq-8.1}\bigl|\{u\ge u(x_0,t_0)/2\}\cap
\rB^{\dd}_r(\check x,\check t)\}\bigr|\le\biggl|\f{2}{u(\check
x,\check t)}\biggr|^{\vep_*}\le C
s_0^{-\vep_*}|\rB^{\dd}_r|.\end{equation} In order to get a
contradiction, we estimate $|\{u\le u(\check x,\check t)/2\}\cap
\rB^{\dd}_{\dt r/2}(\check x,\check t)|$ for some very small $\dt>0$
(to be determined later). For any $(x,t)\in\rB^{\dd}_{2\dt r}(\check
x,\check t)$, we have that $u(x,t)\le
s_0(\dd_0-\dt\dd_0)^{-\bt}=u(\check x,\check t)(1-\dt)^{-\bt}$ for
$\dt>0$ so that $(1-\dt)^{-\bt}$ is close to $1$. We consider the
function
$$v(x,t)=(1-\dt)^{-\bt}u(\check x,\check t)-u(x,t).$$ Then we see that $v\ge 0$ on
$\rB^{\dd}_{2\dt r}(\check x,\check t)$, and also $\cM_{\fL}^-
v-v_t\le 1$ on $Q_{\dt r}(\check x,\check t)$ because $\cM_{\fL}^+
u-u_t\ge -1$ on $Q_{\dt r}(\check x,\check t)$. We now want to apply
Theorem \ref{thm-harnack-weak} to $v$. However $v$ is not positive
on $\BR^n$ but only on $Q_{\dt r}(\check x,\check t)$. To apply
Theorem \ref{thm-harnack-weak}, we consider $w=v^+$ instead of $v$.
Since $w=v+v^-$, we have that $\cM_{\fL}^- w-w_t\le\cM_{\fL}^-
v-v_t+\cM_{\fL}^+ v^--v^-_t\le 1+\cM_{\fL}^+ v^--v^-_t$ on $Q_{\dt
r}(\check x,\check t)$. Since $v^-\equiv 0$ on $\rB^{\dd}_{2\dt
r}(\check x,\check t)$, if $(x,t)\in Q_{\dt r}(\check x,\check t)$
then we have that $\mu(v^-,x,y,t)=v^-(x+y,t)+v^-(x-y,t)$ for $y\in\BR^n$.

Take any $(x,t)\in Q_{\dt r}(\check x,\check t)$ and any
$\vp\in\rC^2_{Q_{\dt r}(\check x,\check t)}(v^-;x,t)^+$. Since
$(x,t)+B_{\dt r}\subset Q_{2\dt r}(\check x,\check t)$ and
$v^-(x,t)=0$ we have that
\begin{equation*}\begin{split}\cM_{\fL}^+
v^-(x,t)-v^-_t(x,t)&=(2-\sm)\int_{\BR^n}\f{\Ld\mu^+(v^-,x,y,t)-\ld\mu^-(v^-,x,y,t)}{|y|^{n+\sm}}\,dy\\
&\le(2-\sm)\Ld\|u\|_{L^{\infty}( \BR^n\times I)}\int_{\{y\in\BR^n:v(x+y,t)<0\}}\f{1}{|y|^{n+\sm}}\,dy\\
&\le (2-\sm)\Ld\|u\|_{L^{\infty}( \BR^n\times I)}\int_{B^c_{\dt
r}}\f{1}{|y|^{n+\sm}}\,dy\\
&\leq  C\f{(2-\sm)\Ld\|u\|_{L^{\infty}( \BR^n\times I)}}{(\dt
r)^{n+\sigma}}
\end{split}\end{equation*}

 Thus  we obtain that $w$ satisfies
$$\cM_{\fL}^- w(x,t)-w_t\le C(\dt
r)^{-\sm}\,\,\text{ on $Q_{\dt r}(\check x,\check t)$ }$$ in
viscosity sense. Since $u(\check x,\check
t)=s_0\dd_0^{-\bt}=2^{-\bt}s_0 r^{-\bt}$ and $\bt\vep_*=\sigma$,
applying Theorem \ref{thm-harnack-weak} we have that
\begin{equation*}\begin{split}&\bigl|\{u\le u(\check x,\check t)/2\}\cap\rB^{\dd}_{\dt
r/2}(\check x,\check t)\bigr|\le\bigl|\{u\le u(\check x,\check
t)/2\}\cap Q_{\dt
r/2}(\check x,\check t)\bigr|\\
&\qquad\qquad\qquad=\bigl|\{w\ge u(\check x,\check
t)((1-\dt)^{-\bt}-1/2)\}
\cap Q_{\dt r/2}(\check x,\check t)\bigr|\\
&\qquad\qquad\qquad\le C(\dt
r)^{n+\sigma}\bigl[((1-\dt)^{-\bt}-1)u(\check x,\check t)+C(\dt
r)^{-\sm}(\dt
r)^{\sm}\bigr]^{\vep_*}\\&\qquad\qquad\qquad\qquad\qquad\qquad\qquad\qquad\times\bigl[u(\check x,\check t)((1-\dt)^{-\bt}-1/2)\bigr]^{-\vep_*}\\
&\qquad\qquad\qquad\le C(\dt
r)^{n+\sigma}\bigl[((1-\dt)^{-\bt}-1)^{\vep_*}+\dt
^{-n\vep_*}s_0^{-\vep_*}\bigr].\end{split}\end{equation*} We now
choose $\dt>0$ so small enough that $C(\dt
r)^{n+\sigma}((1-\dt)^{-\bt}-1)^{\vep_*}\le |\rB^{\dd}_{\dt
r/2}|/4.$ Since $\dt$ was chosen independently of $s_0$, if $s_0$ is
large enough for such fixed $\dt$ then we get that $C(\dt
r)^{n+\sm}\dt ^{-n\vep_*}s_0^{-\vep_*}\le |\rB^{\dd}_{\dt r/2}|/4.$
Therefore we obtain that $\bigl|\{u\le u(\check x,\check
t)/2\}\cap\rB^{\dd}_{\dt r/2}(\check x,\check t)\bigr|\le
|\rB^{\dd}_{\dt r/2}|/2.$ Thus we conclude that
\begin{equation*}\begin{split}\bigl|\{u\ge u(\check x,\check t)/2\}\cap\rB^{\dd}_r(\check x,\check t)\bigr|&\ge\bigl|\{u\ge u(\check x,\check t)/2\}\cap
\rB^{\dd}_{\dt r/2}(\check x,\check t)\bigr|\\&\ge\bigl|\{u>u(\check
x,\check t)/2\}\cap\rB^{\dd}_{\dt r/2}(\check x,\check
t)\bigr|\\&\ge\bigl|\rB^{\dd}_{\dt r/2}(\check x,\check
t)\bigr|-\bigl|\rB^{\dd}_{\dt r/2}\bigr|/2\\&=\bigl|\rB^{\dd}_{\dt
r/2}\bigr|/2=C |\rB^{\dd}_r|,\end{split}\end{equation*} which
contradicts \eqref{eq-8.1} if $s_0$ is large enough. Thus we
complete the proof. \qed

\noindent{\bf Acknowledgement.} Ki-Ahm Lee was supported by Basic
Science Research Program through the National Research Foundation of
Korea(NRF)  grant funded by the Korea
government(MEST)(2010-0001985).

\end{document}